\documentclass[11p]{article}
\usepackage{graphicx} 
\usepackage{amsmath}
\usepackage{amsthm} 
\usepackage{amsfonts}
\usepackage[english]{babel}
\usepackage{xcolor}

\usepackage{hyperref}
 \usepackage{soul}
\parskip=\medskipamount
\newtheorem{defi}{Definition}
\newtheorem{prop}{Proposition}
\newtheorem{lem}{Lemma}

\newtheorem{cor}{Corollary}

\title{Projective geodesic extensions by conformal modifications in nonholonomic mechanics}

\author{ Malika Belrhazi and Tom Mestdag \\[2mm]
	{\small Department of Mathematics,  University of Antwerp,}\\
	{\small Middelheimlaan 1, 2020 Antwerpen, Belgium}
}

\date{}

\def\nphi{F}
\def\nvarphi{f}
\def\Gammanh{\Gamma^{nh}}
\def\redvf{\Gamma^{red}}
\def\alphabar{\alpha}
\def\ort{(V\pi,{\mathcal D})}
\def\h{{k}}

\begin{document}

\maketitle

\begin{abstract} Projective geodesic extensions are reparametrizations of the trajectories of a nonholonomic mechanical system (with only a kinetic energy Lagrangian), in such a way that they can be interpreted as part of the geodesics of a Riemannian metric. We derive necessary and sufficient conditions for the existence of these extensions, in the case where the constrained Lagrangian remains preserved up to a conformal transformation. When the nonholonomic system has a symmetry group (a Chaplygin system), we clarify the relation between projective geodesic extensions and closely related concepts, such as $\varphi$-simplicity, invariant measures and Hamiltonization. Throughout the paper, new and relevant examples illustrate the key differences between all these concepts.  
\end{abstract}

\vspace{3mm}

\textbf{Keywords:} Lagrangian systems, nonholonomic constraints, Riemannian metrics, Chaplygin system, geodesic extension, projective transformation, conformal change, invariant measure.

\vspace{3mm}

\textbf{Mathematics Subject Classification:} 
37J06,
37J60,
53Z05,
70G45,
70G65.


\section{Introduction}

The idea that the governing equations of a physics problem should somehow be ``optimal'' or ``minimal'' with respect to the physical input is very prominent in many models of mathematical physics (and of science, in general). For example, in general relativity, the paths that objects follow when influenced only by gravity are geodesics. Roughly speaking, they represent the shortest possible paths the objects can follow in curved spacetime. In the absence of either a potential or external forces, the same can be said about mechanical systems: In that situation, the Euler-Lagrange equations are equivalent with the geodesic equations of the metric that defines the `purely kinetic' Lagrangian. This property remains valid when the system is subjected to holonomic (position-depending) constraints: one may take the holonomic constraints implicitly into account by appropriately shrinking the configuration space. Apart from the inherent beauty of relying on optimal equations, a mathematical model of a physical problem that is based on geodesics has also many other advantages: For example, we may investigate the influence of curvature, exploit the benefits of symmetry, study its integrability, apply geometric numerical integrators, etc. 

In mechanics, the situation changes when the problem includes nonholonomic constraints. These are  non-integrable velocity-dependent restrictions of the motion of the system, such as e.g.\ those caused by wheels (that roll without slipping), skates (that are prevented to move in a direction perpendicular to the blade), pursuit problems, etc. Here, we consider only nonholonomic constraints that can be given in the form of linear (autonomous) equations, and that therefore can be represented by a non-integrable distribution $\mathcal{D}$ on the configuration manifold $Q$ of the mechanical system. Even though most of our results may be extended by including a potential at the appropriate places, we have chosen to limit ourselves here to purely kinetic systems, for the clarity of our exposition.  If $L$ is the (kinetic energy) Lagrangian of the system, the equations of motion for such systems are the so-called {\em Lagrange-d'Alembert equations} of $(L,\mathcal{D})$. They form  a set of first- and second-order ordinary differential equations whose solutions we will call the {\em nonholonomic trajectories} of $(L,\mathcal{D})$. The Lagrange-d'Alembert equations are not Euler-Lagrange equations, and are also not geodesic equations. However, it still remains useful to identify the solutions of these equations as (part of the possibly larger set of) geodesics of some metric  on the configuration space. In the current paper, as in our paper \cite{our}, we investigate methods by which we can determine such a metric. If that can be achieved,  we will speak of  a `geodesic extension' of the nonholonomic dynamics (even though this question seems to be dubbed `the kinetic Lagrangianization  problem' in the conclusions of the paper \cite{Simoes}).

For a geodesic extension, the nonholonomic trajectories of $(L=\frac12 g,\mathcal D)$ are among the geodesics of a newly constructed Riemannian metric $\hat{g}$. A special type of geodesic extension, is the one where $\hat g$ is a so-called {\em $\mathcal{D}$-preserving modification} of the metric $g$. This means that  $\hat{g}$  preserves $g$ on the constraint distribution $\mathcal D$. Under this assumption, we were able to derive in \cite{our}  two conditions (denoted by $(A)$ and $(B)$ in \cite{our}) whose solutions can be used to construct a geodesic extension $\hat g$. Even though this is already a very strong tool, there exists systems for which our method can not be applied (for example, when there is only one nonholonomic constraint, see the discussion in Section~9 of \cite{our}). In order to apply our theory to a broader class of nonholonomic systems, we will present in this paper an interesting generalization, by adding an extra twist. 

After some preliminaries in Section~\ref{sec:prelim}, we introduce in Section~\ref{sec:pge} a new concept: {\em projective geodesic extensions}. Here, we search for a Riemannian metric $\hat{g}$ which is such that its {\em pre}geodesics contain the nonholonomic trajectories, that is:  geodesics of the metric $\hat g$, but possibly with a different parametrization than the one by time. In particular, we wish to accomplish this by means of a {\em $\mathcal{D}$-conformal modification}. This is, in essence, a $\mathcal{D}$-preserving modification, up to an extra conformal transformation by a conformal factor $\nphi$. Under these assumptions, we derive in Lemma~\ref{lem:ApBp} two new conditions (now: $(A')$ and $(B')$), to obtain a valid projective geodesic extension. We illustrate our method with the examples of the generalized  nonholonomic particle  and the two-wheeled carriage. We will continue to use these examples throughout the whole paper, because they neatly demonstrate different aspects of our results.

In the literature, one may already find some earlier results concerning (what we now call) projective geodesic extensions. For example, in \cite{Simoes}, the authors construct one such   extension (in their Theorem~3.4), but only for the case of a {\em Chaplygin system} (a nonholonomic system with a symmetry Lie group) that is, in addition,   {\em $\varphi$-simple} (in the wording of  \cite{Garcia}). In this paper we start from scratch, at the level of an arbitrary kinetic nonholonomic system, and we use a different approach. One of the main goals of this paper is to show that the very specific case of Theorem~3.4 in \cite{Simoes} is not the only one of interest: we will find (both in theory and in examples) projective geodesic extensions for nonholonomic systems that do not adhere to any of these strong assumptions.

If we set the geodesic extension problem in a broader context, it can be regarded as a special case of the kind of {\em Hamiltonization problems} that have been discussed in e.g.\  \cite{BFM,marta}. There, the main idea is to enlarge the mixed first- and second order system of $(L,\mathcal{D})$  to a (chosen) full system of second-order differential equations, and to write those as the Euler-Lagrange equations of a new Lagrangian $\hat L$.  When compared to the setting of this paper, both the given Lagrangian $L$ and the sought-for Lagrangian $\hat L$ are here of purely-kinetic-type (i.e.\ coming from a Riemannian metric and with constant potential).

Unfortunately, the terminology Hamiltonization is over-used in the literature. In other papers such as e.g.\ \cite{BalseiroFernandez,Balseiro,Bol,Ehlers} (just to recall a few), Hamiltonization stands for the approach where one uses the symmetry group $G$ of the (Chaplygin) nonholonomic system to reduce it from a dynamical system on $Q$ to a system on the quotient manifold $Q/G$. This {\em reduced system} is then said to be {\em Hamiltonizable} if one can find a symplectic or Poisson structure on the quotient manifold for which it is Hamiltonian. To  set this apart from the Hamiltonization we discussed earlier, we will call this {\em Hamiltonian reparametrization}. In the context of this type of Hamiltonization, new concepts such as $\varphi$-simplicity, dynamical gauge transformations, and invariant measures have recently been shown to play  a key role \cite{Simoes, Sansonetto,Cantrijn,Garcia, Fedorov,Zenkov}. One of the main strengths of the current paper is that we are able to unify both approaches to Hamiltonization, and that we can clarify the interplay between these various subtly different notions and concepts. 

For that purpose, we further explore the applicability of the conditions $(A')$ and $(B')$, in the event of a  Chaplygin system in Section~\ref{chap}. Because of the symmetry properties of this type of nonholonomic systems, we can simplify these conditions (see Prop \ref{CondB}) and at the end of the section we show how linear first integrals can be employed. The technique is then applied to the generalized  nonholonomic particle. Throughout the paper, special attention is given to the subclass of geodesic extensions $\hat g$ which have the property that the vertical distribution $V\pi$ of the principal fibre bundle $\pi: Q\to Q/G$ is orthogonal to the constraint distribution $\mathcal D$. We call these type of conformal geodesic extensions {\em $\ort$-orthogonal.} The above mentioned result of \cite{Simoes} falls into this category. In Corollary~\ref{corollary} we show that, in that case, condition $(B')$ is even redundant. This explains why condition $(B')$ does not seem to have made any appearance yet in the literature: the assumption of  $\ort$-orthogonality is often implicitly assumed.

We have already mentioned that some of the well-known examples of Chaplygin systems possess the $\varphi$-simplicity property. In Section~\ref{pges}, we investigate the relation between this property and our condition $(A')$, first at the level of the full configuration space $Q$. We use the examples to demonstrate that the class of conformal geodesic extensions is meaningful: The metrics we had found in Section \ref{firstmodpar} lead to conformal geodesic extensions, that are more general than those that can be obtained through the  $\varphi$-simplicity of the system. 

Because $\varphi$-simplicity is essentially related to Hamiltonian reparametrization of the reduced equations, we then reduce our condition $(A')$ to the quotient space (Lemma~\ref{inv}). We then show in Proposition~\ref{lem:equiv} (and in its special case, Corollary~\ref{lem5}) that we can give an equivalent characterization of our condition $(A')$ that is very reminiscent to (but more general than) the one that can be given for $\varphi$-simplicity (in Proposition~\ref{phi-simpl}). Next, we see that the $\varphi$-simplicity property is only a sufficient condition for the existence of a conformal geodesic extension, and we characterize the obstruction in Proposition~\ref{prop5}.

Hamiltonian reparametrization of the reduced system is also closely related to the existence of an invariant measure (see  e.g.\ \cite{ohsawa,Balseiro, Sansonetto}). In Section~\ref{sec:invmeasure}, we discuss the existence of an invariant measure on the reduced space. We first give a sufficient, and then an equivalent,  condition for the existence of an invariant measure, in terms of the reduced version of $(A')$. For the situation of an $\ort$-orthogonal geodesic extension, we give  in Propositions~\ref{propCantrijn}, \ref{propprojgeodext} and \ref{propphisimple} a full overview of all the conditions that lead to either conformal geodesic extensions, $\varphi$-simplicity  or an invariant measure. It is important to realize that none of the classes in our classification are void: we are able to construct a (mathematical) example of a Chaplygin system that is not $\varphi$-simple, but still allows for a $\ort$-orthogonal projective geodesic extension. 

We end this paper by examining  Hamiltonian reparametrization of the reduced equations. In Proposition~\ref{prop14} (i) we show that, once we have a projective geodesic extension (on the full space), under a further sufficient condition, we are able to give a reparametrization of the solutions of the reduced equations, so that they also become the geodesics of some Riemannian metric on the quotient manifold $Q/G$. This is, therefore, in effect a Hamiltonian reparametrization.  We further show that, under the assumption of $\ort$-orthogonality, the sufficient condition of Proposition~\ref{prop14} (i) is always satisfied. 
In this sense, this also generalizes the result in Theorem~3.4 (in essence their Corollary~3.3) of \cite{Simoes} that we had quoted earlier. We believe that the current approach via geodesic extensions provides an answer to a question that was raised at the end of \cite{Garcia}: to find examples of Hamiltonian reparametrization of Chaplygin systems which are not $\varphi$-simple. This is the case for our mathematical example.

Throughout this paper, the concepts we introduce are globally defined. However, in specific examples one often finds only local expressions. In Section~\ref{localglobal} we discuss some consequences of this in some detail.

\section{Purely kinetic nonholonomic systems}\label{sec:prelim} \label{prelimpk}

Most of these preliminaries (with some more details) can also be found in \cite{our} and are repeated here only to keep the paper self-contained.

In the next sections, we always assume that $Q$ is an $n$-dimensional real manifold, with local coordinates $(q^\alpha)$. The set $\{X_\alpha\}$ of vector fields on $Q$ represents an anholonomic frame, with Lie brackets $[X_\alpha,X_\beta] = R^\gamma_{\alpha\beta}X_\gamma$. We may use it to decompose any tangent vector $v_q\in T_qQ$ as $v_q = v^\alpha X_\alpha (q)$. We  call the components $v^\alpha$ the {\em quasi-velocities} of $v_q$ and we often use $(q^\alpha,v^\alpha)$ as (non-standard) coordinates on the tangent manifold $TQ$.

There are two canonical ways to lift a vector field $X$ on $Q$ to a vector field on $TQ$.  If  $X=X^\alpha\partial/\partial{q^\alpha}$, then in natural coordinates $(q^\alpha, {\dot q}^\alpha)$ on $TQ$, the {\em complete} and {\em vertical} lifts of $X$ are, respectively, given by
$$
 X^C=X^\alpha \frac{\partial}{\partial q^\alpha}+\dot{q}^\beta \frac{\partial X^\alpha}{\partial q^\beta}\frac{\partial}{\partial \dot{q}^\alpha} \quad\text{and}\quad X^V=X^\alpha \frac{\partial}{\partial \dot{q}^\alpha}.
$$

A vector field $\Gamma$ on $TQ$ is a so-called {\em second-order ordinary differential equations field} (SODE, in short) if all its integral curves $\gamma: I\rightarrow TQ$ are of the type $\gamma=\dot{c}$. The corresponding curves  $c:I\rightarrow Q$ are called {\em base integral curves} of $\Gamma$.  
When expressed in terms of the frame $\{X_\alpha^C,X_\alpha^V\}$ of vector fields on $TQ$, a SODE $\Gamma$ is of the form 
$$
\Gamma=v^\alpha X_\alpha^C+F^\alpha X_\alpha^V,
$$
for some functions $F^\alpha\in\mathcal{C}^\infty(TQ)$. The quasi-velocity components $(v^\alpha(t))$ of the integral curves $\gamma(t)$ of $\Gamma$ satisfy
$\dot{v}^\alpha=F^\alpha$.

The main focus of this paper is on  nonholonomic mechanical systems.
\begin{defi}
 A nonholonomic system on a manifold $Q$ consists of a pair $(L,\mathcal{D})$, where $L:TQ\to \mathbb{R}$ is the Lagrangian of the system and $\mathcal{D}$ is an $m$-dimensional distribution on $Q$ that describes the (linear) nonholonomic constraints.  
\end{defi}
If we set $k:=n-m$, the indices $\alpha$ run from 1 to $n$, $a$ from 1 to $m$ and $i$ from 1 to $k$. 

The {\em Lagrange-d'Alembert principle} (see e.g.\ \cite{Bloch,Cortes}) leads to the equations of motion for a nonholonomic system. If $v_q\in \mathcal{D}_q\subset T_qQ$, the  nonholonomic trajectory  $c_{v_q}(t) = q^\alpha(t)$  that starts at $v_q$  is a solution of the differential equations 
\begin{equation}
\begin{cases}
\displaystyle\mu_\alpha^i\dot{q}^\alpha=0,
\\
\displaystyle\frac{d}{dt}\Big(\frac{\partial L}{\partial \dot{q}^\alpha}\Big)-\frac{\partial L}{\partial q^\alpha}=\sum_{i=1}^k{\Lambda_i} \mu^i_{\alpha}.
\end{cases}\nonumber 
\end{equation}
The first equations represent the nonholonomic constraints. In the second equations, the functions $\Lambda_i(t)=\Lambda_i(q(t),\dot{q}(t))$ are called the {\em Lagrange multipliers} of the system. 

We can conceive the nonholonomic trajectories as the base integral curves of a vector field $\Gammanh$. We follow here the approach of e.g.\ \cite{anhol}.  First, assume that the frame $\{X_\alpha\} = \{X_a,X_i\}$ is such that  the vector fields $\{X_a\}$ span ${\rm Sec}(\mathcal{D})$ and that $(v^\alpha)=(v^a,v^i)$ are the quasi-velocities with respect to this frame.  This allows us to characterize the nonholonomic constraints in a simple manner: $v_q\in \mathcal{D}_q$ if and only its corresponding quasi-velocities satisfy  $v^i=0$. Under the regularity condition that the matrix
\begin{equation}
X_a^V(X_b^V(L))\label{regcond}    
\end{equation}
is everywhere on $TQ$ non-degenerate,
the {\em nonholonomic vector field} $\Gammanh$ can be characterized as the unique vector field of the type
$$
\Gammanh=v^a X_a^C+F^a X_a^V,
$$
that satisfies the relations
$$
\Gammanh(X^V_a(L))-X^C_a(L)=0, \qquad \textrm{on }\mathcal{D}.
$$ 
Once $\Gammanh$ has been determined from the above equations, the functions $\lambda_i:=\Gammanh(X^V_i(L))-X^C_i(L)$ on $\mathcal{D}$ play the role of the Lagrangian multipliers (with respect to the current frame).

In what follows, we will only consider a special subclass: that of a Lagrangian that has no (or only constant) potential energy, and a kinetic energy function that comes from a Riemannian metric $g$, i.e.\ a Lagrangian is of the `purely kinetic' type 
\begin{equation} \label{pklag}
L(q,v)=\frac{1}{2}g_q(v,v).
\end{equation}

Starting from a given frame of vector fields, we can always construct a new one, $\{X_a,X_i\}$, that is such that $\{X_a\}$ span ${\rm Sec}(\mathcal{D})$ and  that $\{X_i\}$ span the sections of its orthogonal complement ${\rm Sec}(\mathcal{D}^g)$. In this frame, we use the notations $g_{ab}=g(X_a,X_b)$ and $g_{ij}=g(X_i,X_j)$, while $g(X_a,X_i)=0$. From now on, we will always assume that $\{X_a,X_i\}$ is an orthogonal frame for $\mathcal{X}(Q)$. 

 In the case of a purely kinetic Lagrangian (\ref{pklag}), the regularity condition (\ref{regcond}),  with now $X_a^V(X_b^V(L))=g_{ab}$,  is always satisfied by virtue of the positive-definiteness of the Riemannian metric $g$. The nonholonomic vector field $\Gammanh$ takes an easy form (see e.g.\ Lemma~1 in \cite{our}). If $\nabla$ is the Levi-Civita connection of $g$ and if $\nabla_{X_a} X_b = \Gamma_{ab}^c X_c + \Gamma_{ab}^i {X}_i$, then 
$$
\Gammanh=v^a X_a^C -  \Gamma^a_{bc}v^bv^c X_a^V.
$$
As a consequence, besides $v^i=0$, the remaining quasi-velocities of a nonholonomic trajectory  satisfy 
\begin{equation} \label{nonholeqmotion}
{\dot v}^d = -\Gamma^d_{ab}v^a v^b.
\end{equation}
In \cite{our}, we have also computed that in this case
\begin{equation}\label{lambda}
\lambda_i = g_{ki}\Gamma^k_{ab}v^av^b.
\end{equation}

\section{Projective geodesic extensions} \label{sec:pge}

\subsection{The two principal conditions}

Let $(L=\frac{1}{2}g,\mathcal{D})$ be a purely kinetic nonholonomic system ({\em not} necessarily of Chaplygin type, but as in Section~\ref{prelimpk}). In \cite{our} we have introduced the concept of a {\em geodesic extension} as a tool by which we can interpret the nonholonomic trajectories of pure kinematic Lagrangians as part of the geodesics of a Riemannian metric $\hat{g}$. We have often limited ourselves to the search of a  {\em $\mathcal{D}$-preserving modification}: by this, we mean that  the coefficients $\hat{g}_{ab}=\hat{g}(X_a,X_b)$ of (part of) the sought-for $\hat g$ are equal to those $g_{ab}=g(X_a,X_b)$ of the given kinetic energy metric $g$. It turned out that we could characterize a geodesic extension, solely in terms of conditions on the coefficients $\hat{g}_{ai}=\hat{g}(X_a,X_i)$, and that no further restrictions were required for the coefficients $\hat{g}_{ij}=\hat{g}(X_i,X_j)$: $\hat{g}_{ai}$ represents a $\mathcal D$-preserving geodesic extension when $\theta_k:=\hat{g}_{ak}v^a$ satisfies
\begin{eqnarray*} 
(A) && \theta_k R^k_{ac} v^a =0,\\
  (B) &&  \Gammanh(\theta_i)+\theta_k R^k_{ia}v^a +\lambda_i=0.
	\end{eqnarray*}
In this paper  we  allow for a larger class of geodesic extensions.

A   {\em quadratic spray} (also often called {\em affine spray}) is the SODE that can be associated to a linear connection. The geodesic spray of a Riemannian metric is, in fact, the quadratic spray of its Levi-Civita connection. A quadratic spray is a special case of the more general concept of a spray (see, for example, \cite{Szilasi} for a general reference on the geometry of sprays). Two (quadratic) sprays $\Gamma$ and $\hat\Gamma$ are said to be {\em projectively equivalent} if they have the same base integral curves as point sets. This means that for any base integral curve of $\hat\Gamma$ there exists an orientation preserving reparametrization that transform it into a base integral curve of $\Gamma$. It is shown in e.g.\ \cite{Shen, Szilasi}, that two sprays are projectively related if, and only if, there exists a (positively) 1-homogeneous function $P$  such that
$$
\Gamma = \hat\Gamma +P \Delta.
$$
Here, $\Delta$ stands for the Liouville vector field. When $\Gamma$ is a geodesic spray, the base integral curves of $\hat\Gamma$ are called {\em pregeodesics} of $\Gamma$. We refer to e.g.\  \cite{muzsnay1,muzsnay2} for the use of projective classes of sprays in (Finsler) metrizability problems.

We now use a frame $\{X_\alpha\}$, with quasi-velocities $v^\alpha$. When the sprays are both quadratic
$$  
\Gamma = v^\alpha X_\alpha^C-\Gamma^\alpha_{\beta\gamma}v^\beta v^\gamma X_\alpha^V,
\qquad \hat\Gamma = v^\alpha X_\alpha^C-{\hat\Gamma}^\alpha_{\beta\gamma}v^\beta v^\gamma X_\alpha^V, $$
the function $P$ is a linear function, $P=P_\beta v^\beta$ (see e.g.\ \cite{EM,Mikes}). With  $\Delta = v^\alpha X_\alpha$, the defining relation between projectively equivalent quadratic sprays becomes
\[
\Gamma^\alpha_{\beta\gamma} v^\beta v^\gamma= {\hat\Gamma}^\alpha_{\beta\gamma} v^\beta v^\gamma - P_\beta v^\beta v^\alpha.
\]
Remark that many of our formulae contain factors $v^av^b$. The reason is that in the anholonomic frames $\{X_\alpha\}$ that we use, the Christoffel symbols $\Gamma^\alpha_{\beta\gamma}$ (among other) are not necessarily symmetric in $\beta$ and $\gamma$. In this way, we avoid very lengthy formulas.

In \cite{our} we have called a Riemannian metric $\hat g$  a {\em geodesic extension} of a purely kinetic nonholonomic system $(L,\mathcal D)$ if its  geodesic  spray $\Gamma_{\hat g}$ is related to the nonholonomic vector field $\Gammanh$ by  
$\Gammanh=\left(\Gamma_{\hat{g}}\right)|_{\mathcal D}$. As we have already discussed before, this means that we can interpret the nonholonomic trajectories as those geodesics of the Riemannian metric $\hat g$ whose initial tangent vector lies in the nonholonomic distribution $\mathcal D$. In the light of the previous notions, the following generalization of that definition is now natural:

\begin{defi}
A projective geodesic extension of a purely kinetic nonholonomic system $(L,\mathcal{D})$ is a Riemannian metric $\hat g$ on $Q$ together with a linear function $P\in\mathcal{C}^{\infty}(TQ)$ such that the nonholonomic trajectories are  geodesics of $\hat g$ after a reparametrization (associated to $P$).      
\end{defi}

In practice, we need to look for a couple $(\hat{g}, P)$ such that 
$$
\Gammanh=\left(\Gamma^{\hat{g}}+P\Delta\right )|_{\mathcal{D}}, 
$$
where $\Delta =v^i X_i ^V+v^a X_a ^V$, or also, in components,
\begin{equation} \label{one}
    \begin{cases}
       \Gamma_{bc}^a v^b v^c=\hat{\Gamma}^a_{bc}v^b v^c- P_b v^bv^a,\\
       0=\hat{\Gamma}^i_{bc} v^b v^c.
    \end{cases}
\end{equation}
The last equation says that, on the constraint manifold $v^i=0$, we have $\Gamma_{\hat g}(v^i) = 0$. Notice that when $P|_{\mathcal D}=0$, we obtain a geodesic extension.

We have chosen to use projective classes of sprays to represent reparametrizations of geodesics. This is very common in Riemannian geometry. In the `mechanics' context of Hamiltonization of possibly non-purely-kinetic-Lagrangians, one often resorts to the use of related vector fields to express reparametrizations, and this is also the approach in e.g.\ \cite{Simoes,Garcia}. At the end of the paper, in Proposition~\ref{prop:last}, we will relate both approaches explicitly.

We will now make one further assumption. This assumption is, again, a generalization to the current context of that that we had also made in \cite{our}. There, we have limited the class of the candidate metrics to those that have some special relation to the kinetic energy $g$ of the nonholonomic system: we have called  a symmetric $(0,2)$-tensor field $\overline{g}$  a {\em $\mathcal{D}$-preserving modification} of  $g$, when
  \[
{\overline g}|_{\mathcal{D} \times \mathcal{D}} = g|_{\mathcal{D} \times \mathcal{D}},
\] 
or: ${\bar g}_{ab} = g_{ab}$. Here, we will consider those metrics $\hat g$ whose restriction to the constraints is related to $g$ by a conformal change, in the sense that 
$$
{\hat g}_{ab} =e^{2\nphi} g_{ab},
$$
for some $\nphi\in\mathcal{C}^\infty(Q)$. Another interpretation of the above is that the constrained Lagrangian  remains preserved up to a conformal transformation, ${\hat L}|_{\mathcal D}= e^{2\nphi}L|_{\mathcal D}$. Since we make no further assumptions on the other components of the metric, we may always rewrite  $\hat{g}=e^{2\nphi}\overline{g}$, for some other symmetric (0,2) tensor field $\hat g$ on $Q$. This leads to the following definition:

\begin{defi}
    A symmetric (0,2) tensor field $\hat g$ on $Q$ is a $\mathcal{D}$-conformal modification of a Riemannian metric $g$ on $Q$ if $\hat{g}$ is given by $e^{2\nphi}\overline{g}$ for a function $\nphi\in\mathcal{C}^{\infty}(Q)$ and for a  symmetric $(0,2)$-tensor field  $\overline{g}$ which is a $\mathcal{D}$-preserving modification of $g$.
\end{defi}
The definition translates to the following form for the components:
\begin{equation}
    \begin{cases}
        \hat{g}_{ab}=e^{2\nphi} g_{ab},\\
        \hat{g}_{ai}=e^{2\nphi} \overline{g}_{ai},\\
        \hat{g}_{ij}=e^{2\nphi}\overline{g}_{ij}.
    \end{cases}\nonumber
\end{equation}
The case $\nphi=0$ (or constant) corresponds to that of a $\mathcal{D}$-preserving modification. 

The above assumption has some consequence for the kind of functions $P$ we can consider for our change of parametrization. It is well-known that linear nonholonomic systems with a mechanical (here:  purely kinetic) Lagrangian have conserved energy $E_L=\frac12 g_{\alpha \beta}v^\alpha v^\beta=L$, i.e.\ $\Gamma^{nh}(E_{L})=0$. Besides, also $\Gamma_{{\hat g}}(E_{\hat g})=0$. If we take into account that both $\Gamma^{nh}(v^i)=0$ and $\Gamma_{\hat g}(v^i)=0$  (when $v^i=0$), we get from $\Gammanh=\left(\Gamma_{\hat{g}}+P\Delta\right )|_{\mathcal{D}}$ that $P$ has the property 
\begin{equation}\label{Pform}
P|_{\mathcal D}=\Gamma_{(L,\mathcal D)}(\nphi) = X_d(\nphi)v^d,
\end{equation}
for the same $\nphi$.

In conclusion, instead of looking for a couple $(\hat{g},P)$ to obtain a projective geodesic extension, we will look for a function $\nphi$ and a Riemannian metric $\overline{g}$  such that the set of equations (\ref{one}) is satisfied. It will turn out (in Proposition~\ref{prop1}) that these conditions are independent of the choice of the components $\overline{g}_{ij}$. For this reason, we will also simply call  the couple $({\overline g}_{ai},\nphi)$ a projective geodesic extension (even though, technically, they provide through the freedom in $\overline{g}_{ij}$ a whole class of geodesic extensions).

In, for example, \cite{KN} one may find that, after a conformal change $\hat{g}=e^{2\nphi}\overline{g}$, the Levi-Civita connection of $\hat g$ can be given in terms of the Levi-Civita connection of $\overline{g}$  by the expression
$$ 
\hat{\nabla}_{X} Y=\overline{\nabla}_X Y + X(\nphi) Y+Y(\nphi) X-\overline{g}(X,Y)\overline{\rm{Grad}}\,\nphi.
$$
If $\overline{g}^{\alpha\beta}$ is the inverse of $\overline{g}_{\alpha\beta}=\overline{g}(X_{\alpha},X_{\beta})$, then    $\overline{\rm{Grad}}\,\nphi=\overline{g}^{\alpha\beta} X_{\alpha}(\nphi)X_\beta$. From this, we get the relevant Christoffel symbols with respect to the frame $\{X_\alpha\}$ of the  Levi-Civita connection of $\hat g$ in terms of those of $ \overline{g}$:
\begin{eqnarray*}
    \hat{\Gamma}^a_{bc} &=&\overline{\Gamma}^a_{bc}  + X_c(\nphi)\delta^a_b+X_b(\nphi)\delta^a_c-\overline{g}_{bc}\overline{g}^{\alpha a}X_{\alpha}(\nphi), 
    \\
    \hat{\Gamma}^i_{bc} &=& \overline{\Gamma}^i_{bc} -g_{bc}\overline{g}^{\alpha i} X_{\alpha}(\nphi).
\end{eqnarray*}

This information, together with the special form (\ref{Pform}) of $P$, can now be pasted into equations (\ref{one}). We obtain: 
\begin{equation} \label{two}
    \begin{cases}
      \Gamma^{a}_{bc}v^b v^c=\overline{\Gamma}_{bc}^a v^b v^c + X_c(\nphi)v^c v^a-{g}_{bc}\overline{g}^{\alpha a} X_{\alpha}(\nphi)v^b v^c,\\
      0= \overline{\Gamma}^i_{bc}v^b v^c-g_{bc}\overline{g}^{\alpha i} X_{\alpha}(\nphi) v^b v^c.
    \end{cases}
\end{equation}

The goal of the next Lemma  is to replace the conditions $(A)$ and $(B)$ (for a geodesic extension) to more general conditions $(A')$  and $(B')$ (for a projective geodesic extension), by including the extra terms that are related to the conformal change and the reparametrization.

\begin{lem} \label{lem:ApBp}
Let $g$ be a Riemannian metric and let  $\hat{g}$ be a pseudo-Riemannian metric which is a $\mathcal{D}$-conformal  modification of $g$,  $\hat{g}=e^{2\nphi}\overline{g}$. Assume that ${\overline g}|_{\mathcal{D}^{\overline g} \times \mathcal{D}^{\overline g}}$ is non-degenerate,  where $\mathcal{D}^{\overline{g}}$ stands for the orthogonal complement of $\mathcal{D}$ w.r.t. $\overline{g}$. The conditions (\ref{two}) are equivalent with the conditions 
\begin{eqnarray*} 
(A') &&  g_{bd}\left(\delta_a^d X_c(\nphi)-\delta_c^d X_a(\nphi)\right)v^a v^b+\theta_{k} R^k_{ac}v^a=0,\\
  (B') &&  \Gammanh(\theta_i)+\lambda_i+\theta_{k}(R^k_{i a}+\delta_i^k X_a(\nphi))v^a-g_{ab} X_i(\nphi)v^a v^b =0,
	\end{eqnarray*}
    where $\theta_i=\overline{g}_{ai}v^a$ and $\lambda_i$ are the Lagrange multipliers (\ref{lambda}). \label{equivalentie}
\end{lem}

\begin{proof}
   We start the proof by writing down the Koszul formula for the Christoffel symbols in the orthogonal frame $\{X_a,X_i\}$. Recall first that, for the metric $g$, we have   $g_{ai}=0$. This leads to  
$$
\left\{\begin{array}{l}
2g_{c d} \Gamma_{a b}^dv^a v^b =(2X_a\left(g_{b c}\right)-X_c\left(g_{a b}\right)-2g_{d b} R_{a c}^d) v^av^b,\\
2g_{ki} \Gamma_{a b}^kv^a v^b =(-X_i\left(g_{ab}\right)-2g_{d b} R_{a i}^d) v^av^b.
\end{array}\right.
$$
Likewise,  for  the metric $\overline{g}$, we get
$$
\left\{\begin{array}{l}
2(\overline{g}_{c d} \overline{\Gamma}_{a b}^d+\overline{g}_{ck}\overline{\Gamma}^k_{ab})v^a v^b =(2X_a\left(\overline{g}_{b c}\right) -X_c\left(\overline{g}_{a b}\right)-2\overline{g}_{d b} R_{a c}^d 
-2\overline{g}_{bk}R^k_{ac}) v^av^b,\\
2(\overline{g}_{di} \overline{\Gamma}_{a b}^d+\overline{g}_{ki}\overline{\Gamma}^k_{ab})v^a v^b =(2X_a\left(\overline{g}_{b i}\right) -X_i\left(\overline{g}_{a b}\right)-2\overline{g}_{d b} R_{a i}^d 
-2\overline{g}_{bk}R^k_{ai}) v^av^b.
\end{array}\right.$$
Since $\overline{g}_{ab}=g_{ab}$,  we can derive the following interrelations between the respective Christoffel symbols: 
\begin{enumerate}
    \item[(i)] $\left(g_{cd}\overline{\Gamma}^d _{ab}+ \overline{g}_{ck} \overline{\Gamma}_{ab} ^k  \right)v^a v^b=\left(g_{cd}{\Gamma}^d _{ab}- \overline{g}_{bk} R_{ac} ^k  \right)v^a v^b$,
    \item[(ii)] $\left(\overline{g}_{di} \overline{\Gamma}^d_{ab}+\overline{g}_{ki}\overline{\Gamma}_{ab} ^k\right)v^a v^b=\left( X_a (\overline{g}_{bi})+g_{ki}\Gamma_{ab} ^ k -\overline{g}_{bk} R^k _{ai}\right)v^a v^b$.
\end{enumerate}
These two identities always hold. We will now apply them in our proof of the equivalence of the conditions. 

Suppose first that conditions (\ref{two}) hold. If we plug (\ref{two}) in the first expression (i), we get 
\begin{eqnarray*}
0&=&    -g_{cd} X_a(\nphi) v^a v^d +g_{cd} g_{ba} \overline{g}^{\alpha d} X_{\alpha}(\nphi)v^a v^b + \overline{g}_{ck} g_{ab} \overline{g}^{\alpha k} X_{\alpha}(\nphi) v^a v^b+\overline{g}_{bk} R^k_{ac}v^a v^b.
\end{eqnarray*}

Given that $\overline{g}^{\alpha d} g_{cd}+\overline{g}^{\alpha k}g_{ck}=\delta^\alpha_c$, we get
\begin{eqnarray*} 0&=&
   \left( -g_{cb} X_a(\nphi) +\delta^{\alpha}_c X_{\alpha}(\nphi)g_{ab}+\overline{g}_{bk} R^k_{ac}\right)v^a v^b=\left[ g_{bd}\left(-\delta_c^d X_a(\nphi)+\delta_a^d X_c(\nphi)\right)+\overline{g}_{bk} R^k_{ac}\right]v^a v^b\\
&=&  g_{bd}\left(-\delta_c^d X_a(\nphi)+\delta_a^d X_c(\nphi)\right)v^a v^b+\theta_{k} R^k_{ac}v^a.
\end{eqnarray*}
This is indeed condition $(A')$ for $\nphi$ and $\theta_k$.

If we insert (\ref{two}) in (ii), we get 
\begin{eqnarray*}
    0&=&-\overline{g}_{di}\Gamma_{ab}^d v^a v^b +g_{ki}\Gamma_{ab}^k v^a v^b+X_a(\overline{g}_{bi})v^a v^b-\overline{g}_{bk}R^k_{ai}v^a v^b \\&&\qquad\qquad+X_a(\nphi)\overline{g}_{di}v^a v^d-X_{\alpha}(\nphi)\overline{g}_{di}\overline{g}_{ba}\overline{g}^{\alpha d} v^b v^a -X_{\alpha}(\nphi)\overline{g}_{ki} g_{ab}\overline{g}^{\alpha k}v^a v^b\\
&=& -\overline{g}_{di}\Gamma_{ab}^d v^a v^b +g_{ki}\Gamma_{ab}^k v^a v^b+X_a(\overline{g}_{bi})v^a v^b-\overline{g}_{bk}R^k_{ai}v^a v^b
     +\overline{g}_{bi}X_a(\nphi)v^a v^d -g_{ba}X_i(\nphi)v^b v^a\\
&=& -\overline{g}_{di}\Gamma_{ab}^d v^a v^b +g_{ki}\Gamma_{ab}^k v^a v^b+X_a(\overline{g}_{bi})v^a v^b-\overline{g}_{bk}R^k_{ai}v^a v^b
+\overline{g}_{b\alpha}\left(X_a(\nphi)\delta^\alpha_i -X_i(\nphi)\delta^\alpha_a\right)v^b v^a\\&=&
  \Gammanh(\theta_i)+\lambda_i+\theta_{k}R^k_{i a} v^a
     +\overline{g}_{b\alpha}\left(X_a(\nphi)\delta^\alpha_i -X_i(\nphi)\delta^\alpha_a\right)v^b v^a\\
&=& \Gammanh(\theta_i)+\lambda_i+\theta_{k}(R^k_{i a}+\delta_i^k X_a(\nphi))v^a-g_{ab} X_i(\nphi)v^a v^b .
\end{eqnarray*}
This is condition $(B')$. In the last steps we have used the expression (\ref{lambda}) for $\lambda_i$ and the fact that $\Gammanh(\theta_i) = (X_a(\overline{g}_{bi})-\overline{g}_{di}\Gamma_{ab}^d)v^a v^b$.

The proof of the other direction is, up to some technicalities, very similar to the one given in Lemma~2 of \cite{our}. In this direction,  we make use of the assumption that $ \overline{g}_{|_{\mathcal{D}^{\overline{g}}\times\mathcal{D}^{\overline{g}} }}$ is non-degenerate. To avoid too much repetition, we omit this part of the proof here.
\end{proof}

The next proposition completes the desired equivalence.  
\begin{prop} \label{prop1}
 Consider the equations $(A')$ and $(B')$ from the previous Lemma.
\begin{enumerate} 
\item If $\hat g$ is a Riemannian metric that is a $\mathcal{D}$-conformal modification of $g$ by $\nphi$ and $(\hat{g},P)$ is a projective geodesic extension of $(L,\mathcal{D})$, with $P$ as in (\ref{Pform}), then the restriction ${\overline g}|_{\mathcal{D}\times\mathcal{D}^g}=({\overline g}_{ai})$ and the function $\nphi\in\mathcal{C}^{\infty}(Q)$ satisfy the conditions $(A')$ and  $(B')$.

\item  If a pseudo-Riemannian metric $\overline g$ that is a $\mathcal{D}$-preserving modification of $g$ is such that its restriction ${\overline g}|_{\mathcal{D}\times\mathcal{D}^g}=({\overline g}_{ai})$, together with a  function $\nphi\in\mathcal{C}^{\infty}(Q)$, satisfy the conditions $(A')$ and $(B')$, then there exists a projective geodesic extension of $(L,\mathcal{D})$ by $\nphi$ and by a Riemannian metric ${\hat g}$.
\end{enumerate}
      
\end{prop}
\begin{proof}
  The first result follows immediately from Lemma~\ref{equivalentie}, since for a Riemannian metric $\hat g$, also $\overline g$ is Riemannian, and its restriction to any distribution is non-degenerate. 
  
  For the other direction, we need to show that we can choose the coefficients $\overline{g}_{ij}$ in such a way that the whole matrix  $$\hat{g}=e^{2\nphi}\begin{pmatrix}
     g_{ab}&  \overline{g}_{ai}\\
    \overline{g}_{ia}& \overline{g}_{ij}
    \end{pmatrix}
    $$
    is positive-definite. Since the function $e^{2\nphi}$ is always positive, it remains to show that we can choose $\overline{g}_{ij}$ in such a way that the matrix $\begin{pmatrix}
     g_{ab}&  \overline{g}_{ai}\\
    \overline{g}_{ia}& \overline{g}_{ij}
    \end{pmatrix}$ is always positive-definite. We omit this proof here, because it can be done in a similar fashion as in Proposition~4 of \cite{our}.
\end{proof}

From the above propositions, we may conclude that the search for a conformal geodesic extension via a $\mathcal D$-conformal modification boils down to finding a solution $(\overline{g}_{ai},\nphi)$ of $(A')$ and $(B')$. It is important to realize that, through the further choice of an appropriate $(\overline{g}_{ij})$, one such solution $(\overline{g}_{ai},\nphi)$ may even lead to several distinct conformal geodesic extensions $\hat g$.

We describe two examples below. We will use them as running examples throughout the paper,  to demonstrate different aspects of our results.

\subsection{The generalized  nonholonomic particle} \label{firstmodpar}

We consider on $\mathbb{R}^3$ the nonholonomic system with a Lagrangian function of the form $$L=\frac{1}{2}g(v,v)=\frac{1}{2}(\Dot{x}^2+\dot{y}^2+\dot{z}^2)$$ and with a constraint of the type  
$$ 
\dot{z}=\rho(x,y,z)\dot{x},
$$
where $\rho(x,y,z)$ depends, at least, on $y$ (if not,  the constraint would not be genuinely nonholonomic). We call this example the {\em generalized  nonholonomic particle} because it encompasses the many versions of the so-called nonholonomic particle that can be found in the literature (with, for example,  $\rho(x,y,z)=y$ or $\rho(x,y,z)=xy$). Moreover, it is, up to some physical constants, also a model  for a knife edge on a horizontal plane (with $\rho(x,y,z) = \tan(y)$) and for a Chaplygin sleigh (but, in that last case we need to generalize the Lagrangian).

Given the constraint, we can construct a basis for the distribution $\mathcal{D}$, as well as a basis for its orthogonal complement $\mathcal{D}^g$ with respect to $g$:
$$
\mathcal{D}=\textrm{span}\left\{X_x:=\frac{\partial}{\partial x}+\rho\frac{\partial}{\partial z},X_y:=\frac{\partial}{\partial y}\right\}\quad\textrm{and}\quad\mathcal{D}^g=\textrm{span}\left\{X_z:=\frac{1}{1+\rho^2}\left[\frac{\partial}{\partial z}-\rho\frac{\partial}{\partial x}\right]\right\}.
$$ 
With respect to this frame the metric coefficients $g_{ab}$ are \[
(g_{ab}) = \begin{pmatrix} g_{xx} & g_{xy} \\ g_{yx} & g_{yy} \end{pmatrix}  = \begin{pmatrix} 1+\rho^2     & 0 \\ 0 & 1    \end{pmatrix}
\]
and the corresponding quasi-velocities are, on the constraint $\dot z = \rho \dot x$, 
$$
v_x=\frac{\dot{x} + \rho\dot{z}}{1 + \rho^2} = \dot{x}\qquad\textrm{and}\qquad v_y=\dot{y} .
$$
In our case we have indices $i=z$ and $a=x,y$, and therefore we have 
$$
\theta_z={\overline g}_{xz} v_x + {\overline g}_{yz} v_y.
$$
Since in conditions $(A')$ and $(B')$ only the bracket coefficients $R_{a\gamma}^k$ occur, we only compute $R_{xy}^z, R_{xz}^z$ and $R_{yz}^z$. These bracket coefficients are here: 
$$
R_{zx}^z=\frac{\partial \rho}{\partial z},\quad R_{zy}^z=\frac{\rho \frac{\partial\rho}{\partial y}}{1+\rho^2},\quad R_{xy}^z=-\frac{\partial \rho}{\partial y},\quad R_{yz}^z=\frac{3\rho\frac{\partial \rho}{\partial y}}{(1 + \rho^2)^2}.
$$
With this, we have all the ingredients to construct the conditions $(A')$ and $(B')$. First,
\begin{equation} \label{Amodpar}
    (A')\quad \Leftrightarrow \quad \left[ -{\overline g}_{xz}\frac{\partial\rho}{\partial y} + (1 + \rho^2)\frac{\partial \nphi}{\partial y}  \right]\dot{x}^2+0\,\dot{y}^2 -\left[{\overline g}_{yz}\frac{\partial \rho}{\partial y}+ \frac{\partial \nphi}{\partial x}+\rho\frac{\partial \nphi}{\partial z}\right]\dot{x} \dot{y}=0
\end{equation}
Since both coefficients need to vanish, we can derive an expression for ${\overline g}_{xz}$ and ${\overline g}_{yz}$ in terms of (the derivatives of) $\rho$ and $\nphi$:
$$
{\overline g}_{xz}= \frac{\frac{\partial \nphi}{\partial y}(1 + \rho^2)}{\frac{\partial\rho}{\partial y}} ,\quad
    {\overline g}_{yz}= -\frac{\frac{\partial \nphi}{\partial x}+ \rho\frac{\partial\nphi}{\partial z}}{\frac{\partial\rho}{\partial y}} .$$

The  substitution of these two expressions in $(B')$ results in a lengthy expression. To get a grip on it, we make an ansatz for both $\rho(x,y,z)$ and $\nphi(x,y,z)$:  we suppose that they only depend on $y$. Then, we obtain that $(B')$ can be written as  
$$
  0\,\dot{x}^2+0\,\dot{y}^2 +\left[\frac{(1 + \rho^2)^2\rho'\nphi''+ \left( - (1 + \rho^2)^2 \rho'' +  (\rho')^3 + 2\rho  (1 + \rho^2)(\rho')^2\right)\nphi'+ (1 + \rho^2)^2\rho'(\nphi')^2}{(\rho')^2(1 + \rho^2)} \right]\dot{x}\dot{y}=0.
$$

The remaining coefficient in $\dot x\dot y$ can be interpreted as an ordinary differential equation in $\nphi(y)$. After solving it, we get 
$$ 
\nphi(y)=\frac{1}{2}\ln \left(\frac{(\rho(y)a_2 - a_1)^2}{1 + \rho(y)^2}\right),
$$
where $a_1,a_2$ can be any constants (not zero together). Its corresponding ${\overline g}_{ai}$ is 
$$
{\overline g}_{xz}=\frac{a_1\rho  + a_2}{
a_2\rho - a_1} =-\rho +  a_2\frac{1+\rho^2}{
a_2\rho - a_1}   ,\qquad {\overline g}_{yz}=0.
$$
We will come back to the interpretation of this class of projective geodesic extensions in Section~\ref{secondmodpar} and Section~\ref{thirdmodpar}.

\subsection{The two-wheeled carriage} \label{first2wheeled}

The two-wheeled carriage (see e.g.\ \cite{anhol,LM}) is also an example that fits within our  methodology. We will use here the same notations as in \cite{our}.

\begin{center}
\begin{tabular}{lr}
\begin{minipage}{7cm}
\hspace{-0.7cm}\includegraphics[scale=0.8]{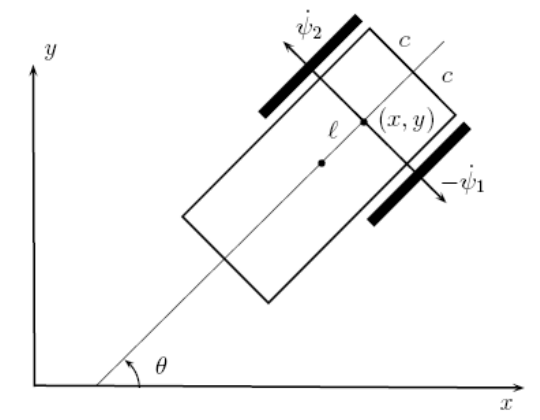}
\end{minipage}
& 
\begin{minipage}{8cm}
The angles  $\psi_1$ and $\psi_2$ describe the rotation of the two wheels (with radius $R$) of the carriage, respectively. We represent  the position of the intersection point of the horizontal symmetry axis of the carriage and the vertical axis that connects the two wheels by the coordinates $(x,y)$. Finally, the angle $\theta$ is the angle between the symmetry axis and the $x$-axis. There are also some relevant parameters that determine the shape of the carriage: $c$ is half of the distance between the two wheels and $\ell$ is  the  distance between the center of mass  and  the  intersection point $(x,y)$. We will think of  $\ell$ as a parameter that we can control when manufacturing the carriage. 
\end{minipage}
\end{tabular}
\end{center}
The configuration space $Q$ is here $S^1 \times S^1 \times
SE(2)$, with Lagrangian function 
$$
L= \frac{1}{2} m ({\dot{x}}^2 + {\dot{y}}^2) +
m_0 \ell {\dot{\theta}}((\cos\theta) \dot{y} - (\sin\theta) \dot{x})
+\frac{1}{2} J {\dot{\theta}}^2 + \frac{1}{2} J_2 ({\dot\psi}_1^2 + {\dot\psi}_2^2 ),
$$
where $J_2$ and $J$ are the moments of inertia for the wheels and for the whole system, respectively, $m_0$ is the mass of the body and $m$ is the mass of the whole system. The two wheels of the carriage are assumed to roll without slipping and this can be described by the following three constraints:
$$
\dot{x} = -\frac{R}{2} \cos\theta (\dot{\psi_1} + \dot{\psi_2}),\quad
\dot{y} = -\frac{R}{2} \sin\theta (\dot{\psi_1} + \dot{\psi_2}), \quad
\dot{\theta} = \frac{R}{2c}  (\dot{\psi}_2 - \dot{\psi}_1).
$$
The vector fields \begin{align*}
X_{\psi_1} &= \frac{\partial }{\partial \psi_1} -
\frac{R}{2}\left(\cos\theta \frac{\partial}{\partial x}+\sin\theta\frac{\partial}{\partial y}+
\frac{1}{c}\frac{\partial}{\partial \theta}\right), \\
X_{\psi_2} &= \frac{\partial}{\partial \psi_2} -
\frac{R}{2}\left(\cos\theta \frac{\partial}{\partial x}+\sin\theta\frac{\partial}{\partial y} -
\frac{1}{c}\frac{\partial}{\partial \theta}\right).
\end{align*}
form a frame $\{X_{\psi_1}, X_{\psi_2}\}$ for the distribution $\mathcal{D}$.

Recall that a geodesic extension  is actually a special case of a projective geodesic extension, with  $\nphi=0$. In \cite{our}, we had found that a (projective) geodesic extension $(\overline{g}_{ai},\nphi)=(\hat{g}_{ai},0)$ exists for this example, with
\begin{equation} \label{gbarone}
\hat{g}_{x{\psi_1}}=\hat{g}_{x{\psi_2}}=\left[ \frac{mR}{2}+\frac{K\ell m_0}{P+Q}\right]\cos\theta,\qquad \hat{g}_{y{\psi_1}}=\hat{g}_{y{\psi_2}}=\left[ \frac{mR}{2}+\frac{K\ell m_0}{P+Q}\right]\sin\theta 
\end{equation}
(with $P= J_2+\displaystyle\frac{R^2m}{4}+\frac{R^2J}{4c}$ and $Q=-\displaystyle\frac{R^2 m}{4}+\frac{R^2 J}{4c}$), but only if $\ell$ has one of the following two specific values:
$$ 
\ell=0\qquad \textrm{or} \qquad \ell=\ell_1:=\frac{\sqrt{(R^2m+2J_2)(JR^2+2J_2c^2)}}{m_0 R^2}.
$$
For other values of $\ell$, there is no geodesic extension (i.e.\ no projective geodesic extension with $\nphi=0$).

\section{Projective geodesic extensions for Chaplygin systems} \label{chap}
\subsection{Chaplygin systems} \label{prelimChap}

For the rest of the paper, we consider {\em Chaplygin systems} (also called generalized Chaplygin systems or nonholonomic systems of principal type, see e.g.  \cite{Bloch,Cantrijn,Garcia}). They are essentially nonholonomic systems with extra symmetry properties.  In case of  a purely kinetic Chaplygin system,  both the Riemannian metric $g$ (and therefore also its Lagrangian $L$) and the constraint distribution $\mathcal{D}$ are invariant under the (free and proper) action $G\times Q \to Q$ of a Lie group $G$ on the configuration manifold $Q$ (when lifted to $TQ$). In addition, we assume that $\pi: Q\to Q/G$ is a principal fibre bundle and that the nonholonomic distribution $\mathcal{D}$ is the horizontal distribution of a principal connection $\omega$ on $\pi$. Even though there exist nonholonomic systems with more general symmetry properties, we will only consider Chaplygin systems here.

Let $\{E_i\}$ be a basis for the Lie algebra of $G$ and let $\{V_i:=(E_i)_Q\}$ be their corresponding infinitesimal generators. They form a basis for ${\rm Sec}(V\pi)$, with $V\pi$ the vertical distribution of $\pi$. If $(q^a)$ are coordinates on $Q/G$, the horizontal lifts $X_a$ of the vector fields $\displaystyle\frac{\partial}{\partial q^a}$ on $Q/G$ (by means of $\omega$) are invariant vector fields on $Q$ that span $\mathcal{D}$. We have
$$
[X_a, V_i]=0,\quad[X_a, X_b]=B_{a b}^j V_j,\quad[V_i, V_j]=-C_{i j}^k V_k,
$$
where $B_{a b}^j$ can be interpreted as components of the curvature of the connection $\omega$ and $C_{i j}^k$ as the structure constants of the Lie algebra.

From now on, we use the notation $G_{ab} = g(X_a,X_b)=g_{ab}$, $G_{ai} = g(X_a,V_i)$ and $G_{ij} = g(V_i,V_j)$ for the components of the kinetic energy metric $g$ with respect to the frame $\{X_a, V_i\}$. We can also construct an orthogonal basis $\{X_a, {X}_i\}$, by setting
\[
X_i=V_i- {g}^{ab}{G}_{ai}X_b=V_i+ {K}^{b}_{i}X_b,\qquad \mbox{with $K^b_i:=- {g}^{ab}{G}_{ai}$},
\] 
where we mean by $g^{ab}$ the inverse of the matrix $g_{ab}$. For this basis $\{X_a,X_i\}$, we have $g_{ab} = g(X_a,X_b)=G_{ab}$ and  $g_{aj} =g(X_a,X_j)=0$ (and $g_{ij} = g({X}_i,{X}_j)$). For the bracket relations in this frame, one may easily verify that  
$$
R^j_{ab}=B_{ab}^j, \qquad R_{ab}^c=-K_j ^c R_{ab}^j, \qquad\mbox{and}\qquad  R_{a i}^j= {K}^{b}_{i} B_{a b}^j. 
$$
From the middle, we may derive a very useful identity:
\begin{equation}\label{useful}
 g_{ce}R^c_{ab}=G_{ej}R^j_{ab}.
 \end{equation}

In  many papers (see e.g. \cite{Bloch,Cantrijn,Garcia}) it has been shown that the nonholonomic equations (\ref{nonholeqmotion}) of a Chaplygin system can be reduced to a set of second-order ordinary differential equations on the quotient manifold $Q/G$. We can easily re-establish this here. When the kinetic energy metric $g$ is $G$-invariant, it defines a reduced metric $g^{red}$ on $Q/G$, by
\begin{equation} \label{gred}
g^{red}\left(\frac{\partial}{\partial q^a},\frac{\partial}{\partial q^b}\right) := g(X_a,X_b)=g_{ab}.
\end{equation}
In \cite{our}, we have used the Koszul formula for an expression of  $\Gamma^c_{ab}$ (components along the orthogonal frame $\{X_a,X_i\}$):
$$
g_{cd}\Gamma^d_{ab}v^av^b = \left(X_a(g_{bc}) - \frac12X_c(g_{ab})  - g_{db}R^d_{ac}\right)v^av^b.
$$
After symmetry reduction (with $v^a={\dot q}^a$), the first two terms in the right hand side can be associated to the Christoffel symbols $\Gamma^{c,red}_{ab}$ of the reduced metric $g^{red}$ with respect to the standard basis $\displaystyle\left\{ \frac{\partial}{\partial q^a}\right\}$ of vector fields on $Q/G$.  Since the equations (\ref{nonholeqmotion}) are equivalent with
$$
g_{cd}{\dot v}^d = -g_{cd}\Gamma^d_{ab}v^a v^b,
$$
the reduced equations become
$$
g_{cd}\left({\ddot q}^d +\Gamma^{d,red}_{ab}{\dot q}^a{\dot q}^b\right)=  g_{db}R^d_{ac}{\dot q}^a{\dot q}^b.
$$
These are the equations that determine the integral curves of the so-called {\em reduced SODE $\redvf$ on $Q/G$},
$$
\redvf = {\dot q}^e \frac{\partial}{\partial q^e} +  \left( -\Gamma^{e, red}_{ab}       + g^{ce}g_{db}R^d_{ac}  \right)   {\dot q}^a {\dot q}^b \frac{\partial}{\partial {\dot q}^e}.
$$
 We recognize in them the coordinate-independent description that appears in e.g.\ \cite{Cantrijn}: $\redvf$  
 is determined by the symplectic-type equation
 \begin{equation}\label{symplectic}
 \iota_{\redvf}\omega_l-d E_l=\alphabar,
 \end{equation}
where here $l=\frac{1}{2}g_{ab}{\dot q}^a{\dot q}^b$ is the reduced Lagrangian, $\theta_l = g_{ab}{\dot q}^a dq^b$, $E_l=l$, $\omega_l=-d\theta_l$, and where the 1-form $\alphabar$ is called the {\em gyroscopic 1-form}. In local coordinates the semi-basic form $\alphabar$ is given by 
$$
\alphabar =   -g_{db}R^d_{ac}{\dot q}^a{\dot q}^b dq^c  = -
R_{ac}^j\dot{q}^a \left(V_j^V(L)\right)_{\mathcal D} dq^c,
$$
after recalling the identity (\ref{useful}) and that $\left(V_j^V(L)\right)_{\mathcal D}$ stands for $G_{bj}v^b$ (when we set $G_{jb} = G_{bj}$).

\subsection{The formulation on the configuration manifold}

We now specialize our results from Section~\ref{sec:pge} to the case of a Chaplygin nonholonomic system with a purely kinetic Lagrangian.  However, we will make one further assumption: we will assume that the sought-for conformal factor $\nphi$ is a $G$-invariant function on $Q$. Even though there is no a priori reason for this, it will turn out to be very convenient. Not only we can reduce the nonholonomic equations of the  Chaplygin system, but we will show that we can also reduce the conditions for the existence of a geodesic extension.  When $\nphi$ is $G$-invariant, it can be identified with a function $\nvarphi$ on $Q/G$, related to $\nphi$ by $\nphi=\nvarphi\circ\pi$. The infinitesimal characterization of this property is that $V_i(\nphi)=0$.

We leave the condition $(A')$ aside, for now. 
 By using the inherent symmetry of a Chaplygin system, we are able to simplify the condition $(B')$.
\begin{prop}
When $(\overline{g}_{ai},\nphi=\nvarphi\circ\pi)$
satisfies $(A')$ for a Chaplygin system, the condition $(B')$  becomes 
$$\Gammanh(C_i)+\Gammanh(\nphi)C_i=0,$$
where $C_i:=C_{ai}v^a:=(\overline{g}_{ai}+G_{ai})v^a$.\label{CondB}
\end{prop}

\begin{proof}
We have shown in \cite{our} that the Lagrangian multipliers, in the case of a Chaplygin system, are of the form $$\lambda_i=\Gammanh\left(G_{ai} v^a\right).
$$
Therefore, the first two terms in condition $(B')$ may be rewritten as 
$$
\Gammanh(\theta_i)+\lambda_i=\Gammanh(\theta_i+G_{ai}v^a)=\Gammanh(\overline{g}_{ai}v^a+G_{ai}v^a) = \Gammanh(C_i).
$$

 The last terms in condition $(B')$ can also be simplified:    \begin{eqnarray*} 
&&  \theta_{k}R^k_{i a} v^a
     +\overline{g}_{b\alpha}\left(X_a(\nphi)\delta^\alpha_i -X_i(\nphi)\delta^\alpha_a\right)v^b v^a\\
    &=& \overline{g}_{bk}R^k_{ia}v^av^b-\overline{g}_{ba}X_i(\nphi)v^av^b+\overline{g}_{bi}X_a(\nphi)v^av^b\\&  =& -\overline{g}_{bk}K_i^dR^k_{ad}v^av^b-\overline{g}_{ba}K_i^d X_d(\nphi)v^av^b+\overline{g}_{bi}X_a(\nphi)v^av^b\\
        &=& -K_i^d(\overline{g}_{bk} R^k_{ad}+\overline{g}_{ba}X_d(\nphi))v^av^b+\overline{g}_{bi}X_a(\nphi)v^av^b= -K_i^d(g_{bd}X_a(\nphi))v^a v^b+\overline{g}_{bi}X_a(\nphi)v^av^b\\
        &=& (G_{bi} +\overline{g}_{bi})X_a(\nphi)v^av^b=  C_{bi} X_a(\nphi)v^av^b = C_i X_a(\nphi)v^a = C_i \Gammanh(\nphi).
    \end{eqnarray*}
In the second step we have used that $R_{ia}^k=K_i^dR^k_{da}$ and that  $ X_i(\nphi)=V_i(\nphi)+K_i^b X_b(\nphi)=K_i^b X_b(\nphi)$. In the fourth step we have applied the condition $(A')$. In the final step, the factor $X_a(\nphi)v^a$ has been rewritten as $\Gammanh(\nphi)$.

When we take both parts together we obtain the desired expression.
\end{proof}

For later use, we state the following easy property: 
\begin{lem} \label{lem2}
    The metric $\overline{g}$ satisfies $\overline{g}|_{V\pi\times \mathcal{D}}=0$ if, and only if, $\overline{g}_{ai}=-G_{ai}$.
    \end{lem}
  \begin{proof}
The assumption $\overline{g}|_{V\pi\times \mathcal{D}}=0$ means that for all $V_i$ and for all $X_a$ we have $\overline{g}(X_a,V_i)=0$, and therefore, indeed,
\begin{eqnarray*}
0&=&\overline{g}(X_a,V_i) =\overline{g}(X_a, X_i-K_i^b X_b)=\overline{g}_{ai}+G^{cb}G_{ci}{g}_{ab}= \overline{g}_{ai}+G_{ai}.  
\end{eqnarray*}
\end{proof}

    \begin{cor} \label{corollary}
    If the couple $(\overline{g}_{ai}=-G_{ai},\nphi=\nvarphi\circ\pi)$ satisfies condition $(A')$, then condition $(B')$ will also be satisfied (i.e.\ the couple $(\overline{g}_{ai}=-G_{ai},\nphi=\nvarphi\circ\pi)$ becomes a projective geodesic extension).
\end{cor}

\begin{proof}
When  $\overline{g}_{ai}=-G_{ai}$, also $C_i=0$.
From Proposition~\ref{CondB}, we conclude that  both terms in condition $(B')$ will vanish. 
\end{proof}

In the next sections, we will extensively discuss the case that is described in  Corollary~\ref{corollary}. We will denote the condition $(A')$ by $(A')^G$, after the substitution $\overline{g}_{ai}=-G_{ai}$. Then, $\nvarphi$ is the only unknown of $(A')^G$.
We also introduce some new terminology.

\begin{defi} A projective geodesic extension of a Chaplygin system via a $\mathcal D$-conformal modification $({\bar g}_{ai},\nphi=\nvarphi\circ\pi)$ is $\ort$-orthogonal when $\overline{g}|_{V\pi\times \mathcal{D}}=0$ (or: $\overline{g}_{ai}=-G_{ai}$, or: $C_i=0$). 
\end{defi}

Corollary~\ref{corollary} then says that $(A')^G$ is sufficient for obtaining a $\ort$-orthogonal conformal geodesic extension. It is, however, also important to realize that this is not the only choice, and that it is often possible to find solutions for condition $(B')$ where $C_i$ is not necessarily zero. 

One possible method to proceed, is as follows. First, we look for at least one non-zero solution $\psi=\psi_av^a$ of the  differential equation $\Gammanh(\psi)+\Gammanh(\nphi)\psi=0$. Then, we  can take one $C_k=\psi$ (or: ${g}_{ak}=-G_{ai}+\psi_a$) and all the other $C_i=0$, for $i\neq k$. 

The function $\psi\neq 0$ needs to satisfy:
$$
\frac{\Gammanh(\psi)}{\psi}+\Gammanh(\nphi)=0\qquad \Leftrightarrow \qquad  \Gammanh\left(\ln{|\psi|}+\nphi\right)=0.
$$
This means that we can relate the solutions $\psi$ to  first integrals $\mu$ of the vector field $\Gammanh$ by
$$
\ln{|\psi|}+\nphi=\mu,
$$
from which  $|\psi|=e^{\mu- \nphi}=e^{-\nphi} e^\mu$. 

Since $\psi$ needs to be linear in $v^a$, the factor $e^\mu$ needs to be linear in $v^a$ as well. Therefore, we are looking for first integrals $\mu$ that are of the form $\mu=\ln(\nu_a v^a)$. It is easy to see that $\mu$ is a first integral of $ \Gammanh$ if and only if  $\nu_a v^a$ is one, that is: $ \Gammanh(\nu_a v^a)=0$.

In conclusion:  we are looking for a (positive) first integral of  $\Gammanh$,  of the form $\nu_a v^a$, and we use any $\psi$ with  
$$ 
|\psi|=e^{-\nphi} \nu_a v^a.
$$

In the next paragraph, we give an example where we specialize the method to the case of one nonholonomic constraint. Be aware, however, that if we can find an $\psi$ in the above way, we still need to check  whether we can obtain a projective geodesic extension $(\overline{g}_{a1} = -G_{a1} + e^{-\nphi}  \nu_a,\nphi)$ by imposing that the couple also satisfies $(A')$.

\subsection{The generalized  nonholonomic particle}\label{secondmodpar}

 When the function $\rho(x,y,z)$ of the generalized  nonholonomic particle does not depend on the coordinate $z$, it is in fact a Chaplygin system, and we can illustrate our method with it. The action of its symmetry Lie group $G=\mathbb{R}$ is given by the translations on $z$. The reduced space $Q/G$ is $\mathbb{R}^2$ with coordinates $(x,y)$. The vertical distribution $V\pi$ is spanned by $\{V_z\}=\{\frac{\partial}{\partial z}\}$ and we get that $G_{xz}=g(X_x,V_z)=\rho(x,y)$ and $G_{yz}=g(X_y,V_z)=0$.

If we enforce the symmetry on all involved objects, the functions $\rho$, $\nphi=\nvarphi$, $\overline{g}_{xz}$ and $\overline{g}_{yz}$ may only depend on the coordinates $(x,y)$. We show first that we can find linear first integrals of the type $\nu_a v^a=\nu_x(x,y)\dot{x}+\nu_y(x,y)\dot{y}$. 

The nonholonomic equations of motion are given by 
$$\begin{cases}
    \dot{z}=\rho\dot{x},\\
    \Ddot{x}=\displaystyle\frac{-\rho\dot{\rho}}{1+\rho^2}\dot{x},\\
    \Ddot{y}=0,
\end{cases}
$$
where a notation like $\dot\rho$ is short for $\displaystyle \frac{\partial \rho}{\partial x}\dot x + \displaystyle \frac{\partial \rho}{\partial y}\dot y$. The condition $\Gammanh( \nu_x\dot{x}+\nu_y\dot{y})=0$ becomes here 
$$ 
\dot{\nu}_x\dot{x}+\dot{\nu}_y\dot{y}-\frac{\nu_x\rho\dot{\rho}}{1+\rho^2}\dot{x}=0.
$$
In view of  the independence between $\dot{x}^2,\dot{x}\dot{y}$ and $\dot{y}^2$, we deduce that 
$$\begin{cases} \displaystyle
    \frac{\partial \nu_x}{\partial x}-\displaystyle\frac{\nu_x\rho}{1+\rho^2}\frac{\partial \rho}{\partial x}=0,
    \\[2mm]
    \displaystyle \frac{\partial \nu_x}{\partial y}
 +\frac{\partial \nu_y}{\partial x}-\displaystyle\frac{\nu_x \rho}{1+\rho^2}\frac{\partial \rho}{\partial y}=0,\\[2mm]
   \displaystyle \frac{\partial \nu_y}{\partial y}=0.
\end{cases}
$$
This implies that $\nu_x(x,y)=\kappa_1(y)\sqrt{1+\rho^2(x,y)}$ and $\nu_y(x,y)=\kappa_2(x)$, where the functions $\kappa_1$  and $\kappa_2$ are related by $\displaystyle\frac{\partial\kappa_1}{\partial y}=\frac{-1}{\sqrt{1+\rho^2}} \frac{\partial\kappa_2}{\partial x}$. In particular, we may consider the special case where $\kappa_1$ and $\kappa_2$  are both  constants. 

Since we have only one constraint here, we try to find a projective geodesic extension with 
$$ 
\overline{g}_{xz}=e^{-\nphi} \nu_x-G_{xz}=\kappa_1e^{-\nphi}\sqrt{1+\rho^2}-\rho\quad\quad\textrm{and}\quad\quad \overline{g}_{yz}=e^{-\nphi} \nu_y-G_{yz}=\kappa_2e^{-\nphi}.
$$

Recall the expressions (\ref{Amodpar}) for condition $(A')$. We can also rewrite them as  PDEs in the unknown $\nphi$:
$$ \begin{cases}
    \displaystyle\frac{\partial\nphi}{\partial x}=- \overline{g}_{yz}\frac{\partial \rho}{\partial y},\\[2mm]
    \displaystyle\frac{\partial \nphi}{\partial y}=\frac{1}{1+\rho^2}\overline{g}_{xz}\frac{\partial \rho}{\partial y}.
\end{cases}
\qquad \Leftrightarrow \qquad
 \begin{cases}
    \displaystyle\frac{\partial\nphi}{\partial x}=-\kappa_2 e^{-\nphi}\frac{\partial \rho}{\partial y},\\[2mm]
    \displaystyle\frac{\partial \nphi}{\partial y}=\frac{1}{1+\rho^2}(\kappa_1e^{-\nphi}\sqrt{1+\rho^2}-\rho)\frac{\partial \rho}{\partial y}.
\end{cases}
$$

We consider two easy cases. First, if we take $\kappa_2=0$, then we get from the first equation that $\nphi$ (and therefore also $\rho$)  may only depend on $y$, and the second equation gives  
\begin{equation} \label{sig1}
\nphi(y)=\frac{1}{2}\ln{\left(\frac{(\kappa_3+\kappa_1\rho)^2}{1+\rho^2}\right)},
\end{equation}
for some constant $\kappa_3$ (where $\kappa_1$ and $\kappa_3$ are assumed not both to be zero). In fact, this is the same projective geodesic extension  $(\overline{g}_{ai}, \nphi)$ that we had already encountered in Section~\ref{firstmodpar}, under the identification $\kappa_1=a_2$ and $\kappa_3=-a_1$.

Second, when $\kappa_2\neq 0$ (and if we set now, for convenience, $\kappa_1=0$), the equations for $\nphi$ become 
\begin{equation} \label{sigmaPDE}
\begin{cases}
    \displaystyle\frac{\partial\nphi}{\partial x}=-\kappa_2 e^{-\nphi}\frac{\partial \rho}{\partial y},\\[2mm]
    \displaystyle\frac{\partial \nphi}{\partial y}=-\frac{\rho}{1+\rho^2}\frac{\partial \rho}{\partial y}.
\end{cases}
\end{equation}
 This set of PDEs has a solution when 
$$
\frac{\partial}{\partial y}\left(\kappa_2 e^{-\nphi}\frac{\partial \rho}{\partial y}\right)=\frac{\partial}{\partial x}\left(\frac{\rho}{1+\rho^2}\frac{\partial \rho}{\partial y}\right).
$$
If we  keep the assumption that $\rho$ only depends on $y$, we can interpret this as a SODE in $\rho(y)$,
$$
\frac{\partial\nphi}{\partial y}\rho'- \rho''=0 \quad\Rightarrow\quad  \rho''=\frac{-\rho(\rho')^2}{1+\rho^2}. 
$$
This SODE has a solution whenever we specify $\rho(0)$ and $\rho'(0)$. Even though this solution has no expression in terms of elementary functions, we can still use such a $\rho$ (and its derivative) to write down the corresponding $\nphi$. From the  first equation in (\ref{sigmaPDE}), we obtain that 
 $ \nphi(x,y)=\ln(-\kappa_2\rho'(y) x+ K(y))$ (for some integration function $K(y)$), in the local region where $K(y)>\kappa_2\rho'(y)x$. When we plug this in, in the second equation of (\ref{sigmaPDE}) we get that 
$$
    \frac{1}{-\kappa_2\rho' x+K}(-\kappa_2\rho''x+K')=\frac{-\rho\rho'}{1+\rho^2}.
 $$   
In the above, we may replace the right-hand side with     $\displaystyle \frac{\rho''}{\rho'}$.  The equation then further simplifies to 
$$
\frac{K'}{K}=\frac{\rho''}{\rho'},
$$
from which $K=\kappa_4\rho'$ (with $\kappa_4$ a constant). We conclude that, for any solution of the SODE in  $\rho$, we can find the following  local conformal factor $\nphi$:
\begin{equation} \label{sig2}
\nphi(x,y)=\ln((\kappa_4-\kappa_2x)\rho'(y)),
\end{equation}
in the region where $(\kappa_4-\kappa_2x)\rho'(y)>0$  and modification
$$ 
\overline{g}_{xz}=-\rho\quad\quad\mbox{and}\quad\quad \overline{g}_{yz}= \frac{\kappa_2}{(\kappa_4-\kappa_2x)\rho'}.
$$
Remark that, even though $\rho$ only depends on $y$, we have now found a conformal factor that depends explicitly on $x$. This result is therefore more general than the one we obtained in Section~\ref{firstmodpar}. The nonholonomic particle with, specifically, $\rho(y)=y$ is most often discussed in the literature on geodesic extensions, invariant measures and Hamiltonization, but neither one of the expressions (\ref{sig1}) or (\ref{sig2}) for the conformal factors can be found there in this generality (as far as we are aware).

\section{$\varphi$-simplicity} \label{pges}

\subsection{Relation to projective geodesic extensions on the configuration manifold}

Besides being Chaplygin, some nonholonomic systems possess a further property: they satisfy the so-called $\varphi$-simplicity property (see e.g.\ \cite{Garcia,Garcia2, Simoes}). In this section, we will investigate how this extra property affects the search for a projective geodesic extension (always by means of a $\mathcal D$-conformal modification, in the sense of  our conditions $(A')$ and $(B')$).

For a Chaplygin system, we can use either the Riemanian metric $g$ or the principal fibre bundle structure to  decompose the tangent bundle $TQ$ as
$$
TQ={\mathcal D} \oplus {\mathcal D}^g = {\mathcal D} \oplus V\pi. 
$$
For a vector field $X=Y^aX_a+Y^iX_i = Z^aX_a + Z^iV_i$ on $Q$,  with $Z^a=Y^a+K^a_iY^i$ and $Z^i=Y^i$, we will use 
\begin{equation} \label{projections}
{\mathcal P}(X)=Y^aX_a,\qquad {\mathcal Q}(X)=Y^iX_i,\qquad \tilde{\mathcal P}(X)=Z^aX_a, \qquad \tilde{\mathcal Q}(X)=Z^i V_i
\end{equation}
for the projections, respectively. In this, ${\tilde {\mathcal Q}}: TQ \mapsto V\pi$ is actually the principal connection $\omega$ for which $\mathcal D$ is the horizontal distribution. If we denote its horizontal lift by ${.}^H$, then its curvature  $\mathcal R$ is, in our convention,
$$
{\mathcal R}(\tilde X,\tilde Y) = \tilde{\mathcal Q}([{\tilde X}^H,{\tilde Y}^H]), \qquad \forall \tilde X,\tilde Y \in \mathcal{X}(Q/G).
$$
Then, we can define a (1,2)-tensor field $\mathcal{T}$ on $Q/G$ by 
$$
\mathcal{T}(\tilde X,\tilde Y) := T\pi\left(\mathcal{P}({\mathcal R}(\tilde X,\tilde Y))\right) \qquad \forall \tilde X,\tilde Y \in \mathcal{X}(Q/G).  
$$

 In the literature \cite{Koiller,Cantrijn,Garcia}, this tensor field is called the {\em gyroscopic tensor}. In our frame $\{X_a,X_i\}$ , we obtain (up to $T\pi$) 
 \begin{equation}\label{gyrtens}
 \mathcal{T}\left(\frac{\partial}{\partial q^a},\frac{\partial}{\partial q^b}\right)=\mathcal{P}([X_a,X_b])= B_{ab}^j \mathcal{P}(V_j) =-B^j_{ab}K_j^d X_d=
 R^d_{ab} X_d.
 \end{equation}
It is clear that this vector field is horizontal.  The `$T\pi$' in the definition stands for the
fact that we need to check that it is also invariant, and that we can regard it as a vector field on $Q/G$, for this reason. The Jacobi identity leads to
\begin{eqnarray*}
    0&=& [V_i,[X_a,X_c]]+[X_c,[V_i,X_a]]+[X_a,[X_c,V_i]]\\
    &=& [V_i,[X_a,X_c]]+0+0\\
    &=& V_i(R^k_{ac})X_k+R^k_{ac}[V_i,X_k]+V_i(R^d_{ac})X_d+R_{ac}^d [V_i,X_d]\\
    &=& V_i(R^k_{ac})X_k+R_{ac}^k [V_i,V_k]+R^k_{ac}V_i(K_k^e)X_e+R_{ac}^k K_k^e[V_i,X_e]+V_i(R_{ac}^d)X_d\\
    &=& V_i(R^k_{ac})V_k +V_i(R^k_{ac})K_k^e X_e-C_{ik}^l R^k_{ac} V_l+R^k_{ac}V_i(K_k^e)X_e+V_i(R_{ac}^d)X_d\\
    &=& \left[V_i(R_{ac}^k)-C_{ik}^l R_{ac}^k \right]V_l+\left[ V_i(R_{ac}^k)K_k^e+R_{ac}^k V_i(K_k^e)+V_i(R_{ac}^e)\right]X_e.
\end{eqnarray*}
The first term implies that 
$$
V_i(R^k_{ac})=C^k_{il} R_{ac}^l.
$$
If we  use this  in the second term, we get:  
$$ 
C_{il}^k R_{ac}^k K_k^e-R_{ac}^k K_k^e C^k_{ie}+V_i(R_{ac}^e)=0,
$$
and therefore $$ V_i(R_{ac}^e)=0.$$
It is now easy to see that
$$ 
[V_i,R^d_{ab} X_d]=V_i(R_{ab}^d)X_d+R_{ab}^d[V_i,X_d]=0.
$$

$\varphi$-simplicity is, by definition, a special property of the  tensor field $ \mathcal{T}$.
\begin{defi} 
 A Chaplygin system is $\varphi$-simple if its gyroscopic tensor $\mathcal{T}$  is of the form $\mathcal{T}=-d\varphi\otimes\rm{id}+\rm{id}\otimes d\varphi$, for some function $\varphi$ on $Q/G$.
\end{defi}

 In coordinates, this means that the bracket coefficients $R_{ab}^d$  satisfy
\begin{equation}\label{phisimpleprop}
R^d_{ab}=-\frac{\partial \varphi}{\partial q^a}\delta_b ^d+\frac{\partial \varphi}{\partial q^b}\delta ^d _a.
\end{equation}

To relate this property to conditions $(A')$ and $(B')$, it is important to notice that the latter are conditions in the unknowns $\nphi$ and $\overline{g}_{ai}$, which are functions on $Q$. The $\varphi$-simplicity property depends on the existence of a  function $\varphi$  on the reduced manifold $Q/G$. Before we set the stage for relating them, we make two observations:

\begin{enumerate} \item We have two options for the configuration manifold:
\begin{itemize}
    \item either we consider the $\varphi$-simplicity property in its  equivalent form as a property on $Q$, by lifting it to the expression
    \begin{equation} \label{phisimpleup}
    R^d_{ab}=-X_a(\phi)\delta^d_b+X_b(\phi)\delta_a^d,\quad\textrm{where now}\quad\phi:=\varphi\circ\pi\in\mathcal{C}^{\infty}(Q),
    \end{equation} 
    \item or, we reduce conditions $(A')$ and $(B')$ to conditions on $Q/G$. In that case, we need to show that all involved terms  are $G$-invariant.  
\end{itemize}
\item We  need  to make appropriate choices for both $\nphi$ and $\overline{g}_{ai}$. While it seems natural to consider $\nphi=\varphi\circ \pi$,  we indicate in the next proposition that the $\ort$-orthogonal projective geodesic extensions,
$$
\overline{g}_{ai}=-G_{ai}, \qquad \mbox{with } G_{ai}=g(X_a,V_i),
$$
(that we had already encountered in the previous section) are the ideal candidates. This natural assumption also lies at the basis of the specific choice that has been made in Theorem~3.4 in \cite{Simoes}.  
\end{enumerate}

\begin{prop} \label{phisimpleprojgeodext} Suppose that the nonholonomic system is $\varphi$-simple. Then, the couple $(\overline{g}_{ai}=-G_{ai},\nphi=\varphi\circ \pi)$ satisfies both conditions $(A')$ and $(B')$, i.e.\ we have a $\ort$-orthogonal projective geodesic extension.
  \label{phisimplicity}
\end{prop}

\begin{proof}
In view of (\ref{phisimpleup}), the left hand side of condition $(A')$ becomes with the choice $\nphi=\varphi\circ \pi$:
\begin{eqnarray*}
&& g_{bd}\left(-\delta_c^d X_a(\varphi\circ \pi)+\delta_a^d X_c(\varphi\circ \pi)\right)v^a v^b+\theta_{k} R^k_{ac}v^a\\
    &&\qquad\qquad\qquad=\left[ g_{bd} R^d_{ac}+\overline{g}_{bk}R^k_{ac}\right]v^a v^b=\overline{g}_{b\alpha} R^{\alpha}_{ac} v^a v^b=\overline{g}([X_a,X_c],X_b)v^a v^b\\&&
   \qquad\qquad\qquad=\overline{g}(B^i_{ac} V_i, X_b)v^a v^b= B^i_{ac}\overline{g}(V_i,X_b)v^a v^b.
\end{eqnarray*}
Since we assume that the metric $\overline{g}$ is such that $\overline{g}|_{V\pi\times \mathcal{D}}=0$, we conclude that the combination $(\overline{g},\nphi=\varphi\circ \pi)$ satisfies condition $(A')$. From Corollary \ref{corollary}, with $\nvarphi = \varphi$, we have that condition $(B')$ is then satisfied as well, again in view of  $\overline{g}_{ai}=-G_{ai}$. 
\end{proof}

The conclusion is  that $\varphi$-simplicity induces, in a very specific way, a projective geodesic extension. It also explains why condition $(B')$ does not seem to have made any appearance yet in the literature: the assumption of  $\ort$-orthogonality is often implicitly assumed. The result in Proposition~\ref{phisimpleprojgeodext} encompasses Theorem~3.4 of \cite{Simoes}. In their Remark~3.5  the authors of \cite{Simoes} had already mentioned that their choice is not unique. Now, we  know that the remaining freedom is encoded in condition $(B')$.

Besides,  the next examples show, that one may still obtain conformal geodesic extensions in the case where the system is not $\varphi$-simple. Moreover, even when the system is $\varphi$-simple, one can still obtain a geodesic extension for a conformal factor $\nphi$ that is unrelated to the function $\varphi$ that characterizes the $\varphi$-simplicity of the system. The overall conclusion is that having a projective geodesic extension via a $\mathcal D$-conformal modification is a far more general property than $\varphi$-simplicity.

\subsection{The generalized  nonholonomic particle} \label{thirdmodpar}

    We consider again the case where $\rho$ only depends on $y$.         Since 
    $$
    [X_x,X_y]=-\frac{\partial \rho}{\partial y} X_z-\frac{\rho\rho'}{1+\rho^2} X_x
    $$ 
    and, therefore, $R_{xy}^x=-\displaystyle\frac{\rho\rho'}{1+\rho^2}$ and $R_{xy}^y=0$,  the requirement (\ref{phisimpleup}) of $\varphi$-simplicity is here: 
\begin{equation*}\begin{cases}
        X_x(\phi)=R^x_{xy},\\
        X_x(\phi) =-R_{xy}^y,\nonumber\end{cases}\qquad\Leftrightarrow\qquad\begin{cases}\displaystyle \frac{\partial\phi}{\partial y}=
      -\frac{\rho\rho'}{1+\rho^2}, \\[3mm]
     \displaystyle  -\frac{\partial \phi}{\partial x}+\rho\frac{\partial\phi}{\partial z}=0.\nonumber\end{cases}
    \end{equation*}

We have already discussed the generalized  particle on two occasions. Not withstanding  the fact that, in Section~\ref{firstmodpar},  we had made the ansatz that $\nphi$ only depends on $y$  (which means, in particular, that $\nphi\in\mathcal{C}^\infty(Q/G)$), the conformal factor 
$$ 
\nphi(y)=\frac{1}{2}\ln \left(\frac{(\rho(y) a_2 - a_1)^2}{1 + \rho(y)^2}\right)
$$
we had found does not necessarily satisfy the requirement for $\varphi$-simplicity, for all choices of the constants $a_1$ and $a_2$: Indeed, we have
$$\frac{\partial \nphi}{\partial y}=-\frac{1}{2}\frac{(a_2\rho-a_1)^2}{1+\rho^2}\left[\frac{(a_2\rho-a_1)^2\rho\rho'-2(1+\rho^2)(a_2\rho-a_1)a_2 \rho}{(a_2\rho-a_1)^2}\right]=-\frac{\rho\rho'}{1+\rho^2}+a_2\frac{\rho'}{(a_2\rho-a_1)},$$
and this equals $\displaystyle -\frac{\rho\rho'}{1+\rho^2}$ if and only if $a_2=0$. Therefore, the conformal factor $\nphi$ with $a_2\neq 0$ leads to a  more general projective geodesic extension,  than the one that could have been obtained by using the function 
 $$ 
\varphi(y)=-\frac{1}{2}\ln \left(1 + \rho(y)^2\right),
$$
for which the generalized  particle is  (up to a constant) $\varphi$-simple.

Given that $\displaystyle\frac{\partial\nphi}{\partial x} \neq 0$, the same conclusion holds for the conformal factor $\nphi(x,y)=\ln((\kappa_4-\kappa_2x)\rho'(y))$ that we had obtained in Section~\ref{secondmodpar}.

\subsection{The two-wheeled carriage} \label{localglobal}

Let us also reconsider the example of the two-wheeled carriage, in the current light of $\varphi$-simplicity. For this example, the vertical vector fields are given by 
\[
{V}_x=\frac{\partial}{\partial x}, \quad
{V}_y=\frac{\partial}{\partial y}, \quad
{V}_{\theta}=\frac{\partial}{\partial \theta} - y \frac{\partial}{\partial x} + x\frac{\partial}{\partial y}.
\]
One may then compute that
\begin{eqnarray*}
&& G_{\psi_1x} = -\frac{R}{2}m\cos\theta + \frac{R}{2c}m_0\ell \sin\theta,\qquad G_{\psi_1y} = -\frac{R}{2}m\sin\theta -\frac{R}{2c}m_0\ell \cos\theta,\\
&& G_{\psi_2x} = -\frac{R}{2}m\cos\theta -\frac{R}{2c}m_0\ell \sin\theta,\qquad G_{\psi_2y} = -\frac{R}{2}m\sin\theta +\frac{R}{2c}m_0\ell \cos\theta.
\end{eqnarray*}

Given that 
$$
R_{\psi_1\psi_2}^{\psi_1}=-R_{\psi_1\psi_2}^{\psi_2}=-\frac{R^3 m_0}{4c^2(P+Q)}\ell,
$$
   the $\varphi$-simplicity property requires $\varphi$ to be a solution of the following two partial differential equations 
$$
\frac{\partial \varphi}{\partial \psi_2}=-\frac{R^3 m_0}{4c^2(P+Q)}\ell\quad\mbox{and}\quad\frac{\partial \varphi}{\partial \psi_1}=\frac{R^3 m_0}{4c^2(P+Q)}\ell.
$$
From this, we get that the two-wheeled carriage is $\varphi$-simple, with 
\[
\varphi=\frac{R^3 m_0}{4c^2(P+Q)}\ell\left(\psi_1-\psi_2\right).
\]

Unfortunately, the function $\varphi$ is only locally defined since it is not periodic in $\psi_1$ and $\psi_2$. Together with what we had said about this example in  Section~\ref{first2wheeled}, we now reach the following conclusions:
\begin{itemize}
\item For any $\ell$, the couple $(\overline{g}_{ai}=-G_{ai},\nphi=\varphi\circ\pi)$ gives a local projective geodesic extension of the two-wheeled carriage (according to Proposition~\ref{phisimpleprojgeodext}).

\item When $\ell$ has the special value $\ell_1\neq 0$ we have actually found two distinct projective geodesic extensions: $(\overline{g}_{ai}=-G_{ai},\nphi=\varphi\circ\pi)$ and $(\overline{g}_{ai}=\hat{g}_{ai},\nphi=0)$ with $\hat{g}_{ai}$  as in expressions (\ref{gbarone}). The second is global and unrelated to the $\varphi$-simplicity property of this Chaplygin system.

\item When $\ell=0$, the expressions for $G_{ai}$ overlap with those in (\ref{gbarone}), taking the extra minus-sign into account. In that case also $\varphi=0$, and this is in agreement with $\nphi=0$. We have therefore found only one global projective geodesic extension.

\end{itemize}

\subsection{Reduction of the conditions for a projective geodesic extension}

The examples  confirm that there exist projective geodesic extensions that are non-$\varphi$-simple, but still  $\mathcal D$-conformal modifications (conditions $(A')$ and $(B')$). Now that we have established the greater generality of the current approach (with respect to e.g.\ \cite{Simoes}), we will make a deeper analysis of how the special case of $\ort$-orthogonality,  $(\overline{g}_{ai}=-G_{ai},\nphi=\nvarphi\circ\pi)$,  fits within it. From Corollary~\ref{corollary} we know that, in that case, we only need to examine condition $(A')$ in more detail. In the literature there exist equivalent formulations of e.g.\ $\varphi$-simplicity that make use of the description of a Chaplygin system as a dynamical system on the reduced configuration manifold $Q/G$. For this reason, we start this section by showing how the condition $(A')$  can also be interpreted as a condition on $Q/G$ (under the assumption that we look for invariant solutions). To show that condition $(A')$ is $G$-invariant, we will rely 
 on the following well-known characterization (see e.g.\ \cite{KN,Warner}).    

 \begin{prop} \label{thm:invariance} Suppose that $G$ is a connected Lie group. A tensor  field 
  $t$ is $G$-invariant if and only if  $\mathcal{L}_{V_i}t=0$, for any infinitesimal generator $V_i$.  
\end{prop}

Recall first that, for a Chaplygin system, the kinetic energy metric $g$ and the distribution $\mathcal D$ are $G$-invariant. 

\begin{lem} \label{inv}
    If we assume that the metric $\overline{g}$ and  the function $\nphi$ are $G$-invariant, then the condition $(A')$ is also $G$-invariant and it reduces to a condition on $Q/G$.
\end{lem}

\begin{proof} We use the previous proposition to show that  all the terms of condition $(A')$ are invariant. 
From the assumptions we have 
    $$
\mathcal{L}_{V_i}g=0,\quad\mathcal{L}_{V_i}\overline{g}=0,\quad \mathcal{L}_{V_i}\nphi= V_i(\phi)=0,\quad \mathcal{L}_{V_i} X_a=[V_i,X_a]=0.
    $$

Condition $(A')$ is given by 
$$
\left[g_{bd}\left(-\delta_c^d X_a(\nphi)+\delta_a^d X_c(\nphi)\right)+\overline{g}_{kb} R^k_{ac}\right]v^av^b=0,
$$
but we may use symmetrization in the indices $a,b$ to toss aside the factors $v^av^b$:
$$
 \overline{g}_{bk}R^k_{ac}+\overline{g}_{ak}R^k_{bc}-g_{bc}X_a(\nphi)-g_{ac}X_b(\nphi)+2g_{ab}X_c(\nphi)=0.
$$
 
The last three terms are invariant because $g_{bc}$ and $X_a(\nphi)$ are all invariant functions. Indeed, 
\begin{eqnarray*}
 0&=&(\mathcal{L}_{V_i}g)(X_a,X_b)= {V_i}(g(X_a,X_b))-g(\mathcal{L}_{V_i}(X_a),X_b)-g(X_a,\mathcal{L}_{V_i}(X_b)) \\
 &=& {V_i}(g_{ab})
\end{eqnarray*}
and
$$
    \mathcal{L}_{V_i}(X_a(\nphi))= \mathcal{L}_{V_i}(X_a)(\nphi)+X_a(\mathcal{L}_{V_i}(\nphi)) =0.
$$

It remains to show that the  terms of the type $\overline{g}_{kb} R^k_{ac}$ are also $G$-invariant. 
We have already established in Section~\ref{pges} that $V_i(R^k_{ac})=C^k_{il} R_{ac}^l$. 
 To find an expression for $V_i(\overline{g}_{bk})$, we work as follows. Given that $X_k=V_k+K_k^e X_e$, we get
\begin{eqnarray*}
 0&=& (\mathcal{L}_{V_i}\overline{g})(X_b,X_k)\\
 &=& V_i(\overline{g}_{bk})-\overline{g}([V_i,X_b],X_k)-\overline{g}(X_b,[V_i,X_k])\\
 &=& V_i(\overline{g}_{bk})-\overline{g}(X_b,-C_{ik}^j V_j+V_i(K_k^f)X_f+K_k^f[V_i,X_f])\\
 &=& V_i(\overline{g}_{bk})-\overline{g}_{bf}V_i(K_k^f)+C^j_{ik}(\overline{g}_{bj}+G_{bj})
\end{eqnarray*}
because $\overline{g}(X_b,V_j)=\overline{g}(X_b,X_j-K_j^a X_a)=\overline{g}_{bj}+G_{jc}g^{ac}g_{bc}=\overline{g}_{bj}+G_{bj}$. From this, we deduce that 
$$
V_i(\overline{g}_{bk})=g_{bf}V_i(K_k^f)-C_{ik}^j(\overline{g}_{bj}+G_{bj}).
$$
Together, this gives 
\begin{eqnarray*}
    \mathcal{L}_{V_i}(\overline{g}_{kb} R^k_{ac})&=&V_i(\overline{g}_{kb} R^k_{ac})\\
    &=&V_i(\overline{g}_{kb})R^k_{ac}+\overline{g}_{kb} V_i(R^k_{ac})\\
    &=& (g_{bf}V_i(K_k^f)-C_{ik}^j\overline{g}_{bj}-C_{ik}^j G_{bj})R_{ac}^k +\overline{g}_{bj}C^j_{ik} R^k_{ac}\\
    &=& [g_{bf} V_i(K^f_k)-G_{bj}C_{ik}^j]R_{ac}^k\\
    &=& g_{bf}[V_i(K_k^f)+K_j^fC_{ik}^j]R_{ac}^k.
\end{eqnarray*}
We need to prove that this vanishes. For this we show first that the inverse coefficients $g^{ab}$ of $g_{ab}=g(X_a,X_b)$ are invariant as well. Indeed, since $g_{ab}$ are invariant and $g^{ab}g_{bc}=\delta_c^a$ we get 
$$
0=V_i(\delta_c^a)=V_i(g^{ab}g_{bc})=V_i(g^{ab})g_{bc}+0,
$$
and therefore $V_i(g^{ab})=0$. Moreover, we know that 
\begin{eqnarray*}
    0&=&(\mathcal{L}_{V_i}g)(X_b,V_k)\\
    &=& V_i(G_{bk})-g([V_i,X_b],V_k)-g(X_b,[V_i,V_k])\\
    &=&V_i(G_{bk})-0+C_{ik}^jG_{bj}.
\end{eqnarray*}
Therefore,
$$
V_i(K_k^f)=-V_i(g^{ef}G_{fk})=-V_i(g^{ef})G_{fk}-g^{ef}V_i(G_{fk})=-g^{ef}C_{ik}^j G_{fj}=-K^f_jC_{ik}^j,
$$
and this means that 
$$
\mathcal{L}_{V_i}(\overline{g}_{kb} R^k_{ac})=g_{bf}[V_i(K_k^f)+K_j^fC_{ik}^j]R_{ac}^k=0. 
$$
\end{proof}

In the rest of the paper, we assume that the metric $\overline{g}$ and  the function $\nphi$ are $G$-invariant and that condition $(A')$  reduces  to a condition on $Q/G$. In the last terms both factors $g_{ab}$ and $X_a(\phi)$ are invariant functions on $Q$, on their own.  We will denote their reduced functions on $Q/G$ by $g_{ab}$ and $\displaystyle\frac{\partial \nvarphi}{\partial q^a}$, respectively. In fact, $g_{ab}$ are the components of the reduced metric $g^{red}$ (see (\ref{gred})) with respect to the standard basis of vector fields on $Q/G$. In the terms of the type $\overline{g}_{kb} R^k_{ac}$, the individual factors are not invariant functions. We will, however, keep using the notation $\overline{g}_{kb} R^k_{ac}$ when technically we mean the reduced function of the entire term. With this convention, the reduced version of $(A')$ is
\begin{equation}\label{Aprimered}
(A')\quad\Leftrightarrow\quad \overline{g}_{bk}R^k_{ac}+\overline{g}_{ak}R^k_{bc}-g_{bc}\frac{\partial\nvarphi}{\partial q^a}-g_{ac}\frac{\partial\nvarphi}{\partial q^b}+2g_{ab}\frac{\partial\nvarphi}{\partial q^c}=0.
\end{equation}

Recall from Section~\ref{prelimChap} that the solutions of the reduced equations can be seen as the base integral curves of the vector field $\redvf$ that is characterized  by the expression (\ref{symplectic}),
$$
\iota_{\redvf}\omega_l=d E_l+\alphabar.
$$

The 1-form $\alphabar$ is what we call the {\em gyroscopic 1-form}. It forms the obstruction for the reduced vector field to be a Hamiltonian vector field with respect to the 2-form $\omega_l=-d\theta_l$. In local coordinates, this semi-basic form is given by 
$$
\alphabar=-\overline{R_{bc}^k \dot{q}^b V_k^V(L)} dq^c=-R_{bc}^kG_{ak}\dot{q}^a\dot{q}^b dq^c.
$$

Next to a gyroscopic one-form, also the so-called {\em gyroscopic two-form} has been introduced in \cite{Cantrijn}: it is (up to our convention in the definition of curvature) the  2-form $\Xi^G$.
 \begin{equation} \label{defXiG}
 \Xi^G(\tilde X,\tilde Y)= -g({\rm{\bf T}}^H, {\mathcal{R}}({\tilde X}^H,{\tilde Y}^H) =-g({\rm{\bf T}}^H,\tilde{\mathcal{Q}}[{\tilde X}^H,{\tilde Y}^H]),\qquad\forall \tilde X,\tilde Y\in\mathcal{X}(Q/G), 
 \end{equation}
where $\mathcal R$ is the curvature of the principal connection $\omega$ on $\pi$ (and $\tilde{\mathcal{Q}}:TQ\rightarrow V\pi$ is its vertical projection, defined in (\ref{projections})), and ${\rm{\bf T}}^H = v^a X_a = {\dot q}^aX_a$ is the horizontal lift of the canonical vector field along $\overline\tau:T(Q/G)\rightarrow Q/G$. In local coordinates $\Xi^G$ is given by 
 $$
 \Xi^G=-\frac{1}{2} R^j_{bc} G_{ja} \dot{q}^a dq^b\wedge dq^c.
 $$
 That this does indeed define  a tensor field on $Q/G$ follows from the fact that  the coefficients $R^j_{bc} G_{ja}$ (similar as $R^j_{bc} {\overline g}_{ja}$ before) are invariant functions on $Q$. We have included a subscript $G$ in our notation to emphasize that $\Xi^G$ depends on the given kinetic energy metric $g$, through its coefficients $G_{ai}$. These are the metric coefficients of $g$ with respect to the basis $\{X_a,V_i\}$.
 
One easily verifies that $\Xi^G$ satisfies 
$$ 
\iota_{\redvf}\Xi^G=\alphabar.
$$
In fact, we could have taken any other  SODE than $\redvf$ on $Q/G$, and still have  obtained  that property.

\subsection{Relation to projective geodesic extensions on the reduced manifold}

We can now use these tensor fields to give an equivalent condition for $\varphi$-simplicity. The condition (ii) below appears already in some way in \cite{Cantrijn,Garcia,ohsawa}. With Proposition~\ref{phi-simpl}, we prove that it is in fact, nothing more than the $\varphi$-simplicity property in hidden form.

 \begin{prop} \label{phi-simpl}
    Given $\varphi\in \mathcal{C}^\infty(Q/G)$,  the following statements are equivalent:
    \begin{enumerate}
    \item[(i)] $\varphi\in \mathcal{C}^\infty(Q/G)$ satisfies the $\varphi$-simplicity property (\ref{phisimpleprop}),
        \item[(ii)] $\varphi\in \mathcal{C}^\infty(Q/G)$ satisfies $d\varphi\wedge\theta_l=\Xi^G$.
         \end{enumerate}
\end{prop}
\begin{proof} Since $d\varphi = \displaystyle\frac{\partial \varphi}{\partial q^b} dq^b$ and $\theta_l = g_{ac}{\dot q}^a dq^c$ we have, because of the identity (\ref{useful}),
\begin{eqnarray*}
        d\varphi\wedge\theta_l=\Xi^G&\Leftrightarrow& \frac{\partial \varphi}{\partial q^b}g_{ac}-\frac{\partial \varphi}{\partial q^c}g_{ab}+ R^k_{bc} G_{ka}=0\\
                 &\Leftrightarrow& \frac{\partial \varphi}{\partial q^b}g_{ac}-\frac{\partial \varphi}{\partial q^c}g_{ab}+R_{bc}^d g_{da}=0
                 \\
                 &\Leftrightarrow& \left[\frac{\partial \varphi}{\partial q^b}\delta^d_{c}-\frac{\partial \varphi}{\partial q^c}\delta^d_b-R_{cb}^d \right]g_{da}=0
    \end{eqnarray*}
The expression between $\left[\cdots\right]$ leads to the $\varphi$-simplicity property (\ref{phisimpleprop}).
\end{proof} 

The above result is also mentioned in the proof of Proposition~3.2 of \cite{Simoes}.

We now introduce a new 2-form $\gamma$, for an arbitrary couple $({\overline g}_{ai},\nphi = \nvarphi\circ\pi)$:
$$ 
\gamma=\gamma_{bca}\dot{q}^a dq^b\wedge dq^c\quad\textrm{with}\quad \gamma_{bca}:=\frac{1}{6}(\overline{g}_{bk}R^k_{ac}-\overline{g}_{ck}R^k_{ab}+2\overline{g}_{ak}R^k_{bc}).
$$
This 2-form allows us to give an equivalent condition for $(A')$, that is quite similar to the expression in Proposition~\ref{phi-simpl}.

\begin{prop} \label{lem:equiv}
The following conditions are equivalent:
 \begin{enumerate}
     \item[(i)] The couple $({\overline g}_{ai},\nphi = \nvarphi\circ\pi)$ satisfies condition $(A')$.
     \item[(ii)] The couple $({\overline g}_{ai},\nphi = \nvarphi\circ\pi)$ satisfies 
     $d\nvarphi\wedge \theta_l=\gamma.$
     \end{enumerate}
\end{prop}
\begin{proof}
Suppose first that condition $(A')$ is satisfied (in its reduced form (\ref{Aprimered})). 
Then,  also 
$$
\left(\overline{g}_{bk}R^k_{ac}+\overline{g}_{ak}R^k_{bc}-g_{bc}\frac{\partial\nvarphi}{\partial q^a}-g_{ac}\frac{\partial\nvarphi}{\partial q^b}+2g_{ab}\frac{\partial\nvarphi}{\partial q^c}\right)\dot{q}^a=0.
$$
In the above, we may interchange the indices $b$ and $c$. This gives
$$
\left(\overline{g}_{ck}R^k_{ab}+\overline{g}_{ak}R^k_{cb}-g_{cb}\frac{\partial\nvarphi}{\partial q^a}-g_{ab}\frac{\partial\nvarphi}{\partial q^c}+2g_{ac}\frac{\partial\nvarphi}{\partial q^b}\right)\dot{q}^a=0. 
$$
When we subtract these two equations we get 
$$\left(\overline{g}_{bk}R^k_{ac}-\overline{g}_{ck}R^k_{ab}+2\overline{g}_{ak}R^k_{bc}-3g_{ac}\frac{\partial \nvarphi}{\partial q^b}+3g_{ab}\frac{\partial \nvarphi}{\partial q^c}\right)\dot{q}^a=0 $$
and therefore also $$\left(g_{ac}\frac{\partial \nvarphi}{\partial q^b}-g_{ab}\frac{\partial \nvarphi}{\partial q^c}\right)\dot{q}^a=\frac{1}{3}\left(\overline{g}_{bk}R^k_{ac}-\overline{g}_{ck}R^k_{ab}+2\overline{g}_{ak}R^k_{bc}\right)\dot{q}^a.$$
This expression coincides with the components of equation  $d\varphi\wedge\theta_l=\gamma$. Indeed,
\begin{eqnarray*}
    d\nvarphi\wedge\theta_l\left(\frac{\partial}{\partial q^b},\frac{\partial}{\partial q^c}\right)&=&\gamma\left(\frac{\partial}{\partial q^b},\frac{\partial}{\partial q^c}\right)\\
   \left( \frac{\partial \nvarphi}{\partial q^b}g_{ac}-\frac{\partial \nvarphi}{\partial q^c}g_{ab}\right)\dot{q}^a&=&\left(\gamma_{bca}\dot{q}^a-\gamma_{cba}\dot{q}^a \right) \\
      \left( \frac{\partial \nvarphi}{\partial q^b}g_{ac}-\frac{\partial \nvarphi}{\partial q^c}g_{ab}\right)\dot{q}^a&=&\frac{1}{6}(2\overline{g}_{bk}R^k_{ac}-2\overline{g}_{ck}R^k_{ab}+4\overline{g}_{ak}R^k_{bc} )\dot{q}^a\\
        \left( \frac{\partial \varphi}{\partial q^b}g_{ac}-\frac{\partial \varphi}{\partial q^c}g_{ab}\right)\dot{q}^a&=&\frac{1}{3}(\overline{g}_{bk}R^k_{ac}-\overline{g}_{ck}R^k_{ab}+2\overline{g}_{ak}R^k_{bc} )\dot{q}^a.
\end{eqnarray*}

For the other direction, if  $d\varphi\wedge \theta_l=\gamma$ is satisfied, then for all $\dot{q}^a$, we have 
$$ 
\left( \frac{\partial \nvarphi}{\partial q^b}g_{ac}-\frac{\partial \nvarphi}{\partial q^c}g_{ab}\right)\dot{q}^a=\frac{1}{3}(\overline{g}_{bk}R^k_{ac}-\overline{g}_{ck}R^k_{ab}+2\overline{g}_{ak}R^k_{bc} )\dot{q}^a.
$$
After canceling the factor $\dot{q}^a$ on both sides this becomes 
$$
\frac{\partial \nvarphi}{\partial q^b}g_{ac}-\frac{\partial \nvarphi}{\partial q^c}g_{ab}=\frac{1}{3}(\overline{g}_{bk}R^k_{ac}-\overline{g}_{ck}R^k_{ab}+2\overline{g}_{ak}R^k_{bc}).
$$
Now, if we interchange the indices $a$ and $b$, we get 
$$
\frac{\partial \nvarphi}{\partial q^a}g_{bc}-\frac{\partial \nvarphi}{\partial q^c}g_{ba}=\frac{1}{3}(\overline{g}_{ak}R^k_{bc}-\overline{g}_{ck}R^k_{ba}+2\overline{g}_{bk}R^k_{ac}).
$$
After adding up both equations, we obtain 
$$
g_{bc}\frac{\partial\nvarphi}{\partial q^a}+g_{ac}\frac{\partial\nvarphi}{\partial q^b}-2g_{ab}\frac{\partial\nvarphi}{\partial q^c}=\overline{g}_{bk}R^k_{ac}+\overline{g}_{ak}R^k_{bc},
$$
and this is condition $(A')$. 
\end{proof}

Now, we consider the special case of a $\ort$ orthogonal projective geodesic extension, that is, the metric coefficients $\overline{g}_{ak}$ are such that  $\overline{g}_{kb}=-G_{kb}$. If we insert this in $(A')$, what remains is a condition on  $\nvarphi\in \mathcal{C}^\infty(Q/G)$ only. As mentioned before, we include this in the notation for clarity: 
\begin{equation}\label{AprimeredG}
(A')^G\quad\Leftrightarrow\quad G_{bk}R^k_{ac}+G_{ak}R^k_{bc}+g_{bc}\frac{\partial\nvarphi}{\partial q^a}+g_{ac}\frac{\partial\nvarphi}{\partial q^b}-2g_{ab}\frac{\partial\nvarphi}{\partial q^c}=0.
\end{equation}
The importance of this substitution is that we know from Corollary~\ref{corollary} that the condition $(A')^G$ is equivalent with the property that the couple $(\overline{g}_{ai}=-G_{ai},\nphi=\nvarphi\circ\pi)$ is a projective geodesic extension.

Likewise, we may set $\gamma^G=\gamma^G_{bca}\dot{q}^a dq^b\wedge dq^c$, with
$$ 
 \gamma^G_{bca}=\frac{1}{6}(-G_{bk}R^k_{ac}+G_{ck}R^k_{ab}-2G_{ak}R^k_{bc})= \frac{1}{6}(-G_{bd}R^d_{ac}+G_{cd}R^d_{ab}-2G_{ad}R^d_{bc}),
$$
in view of the identity (\ref{useful}).

Of course, the following corollary trivially holds.
\begin{cor} \label{lem5}
Given  $\nvarphi\in\mathcal{C}^\infty(Q/G)$, the following statements are equivalent:
 \begin{enumerate}
     \item[(i)] $\nvarphi$ satisfies condition $(A')^G$ (i.e.\ the couple $(\overline{g}_{ai}=-G_{ai},\nphi=\nvarphi\circ\pi)$ is a $\ort$-orthogonal projective geodesic extension).
     \item[(ii)] $\nvarphi$ satisfies 
     $d\nvarphi\wedge \theta_l={\gamma}^G$.
 \end{enumerate}
\end{cor}

For what follows, it is important to realize that $\Xi^G \neq \gamma^G$. In order to write down the relation between these two forms, we first rewrite $\gamma$ in the following way:
\begin{eqnarray*}
    \gamma
    &=& \frac{1}{6}(\overline{g}_{bk}R^k_{ac}-\overline{g}_{ck}R^k_{ab}+2\overline{g}_{ak}R^k_{bc})\dot{q}^a dq^b\wedge dq^c\\
    &=& \frac{1}{6}\sum_{b<c}(\overline{g}_{bk}R^k_{ac}-\overline{g}_{ck}R^k_{ab}+2\overline{g}_{ak}R^k_{bc})\dot{q}^a dq^b\wedge dq^c+ \frac{1}{6}\sum_{c<b}(\overline{g}_{bk}R^k_{ac}-\overline{g}_{ck}R^k_{ab}+2\overline{g}_{ak}R^k_{bc})\dot{q}^a dq^b\wedge dq^c\\
    &=&\frac{1}{6}\sum_{b<c}(\overline{g}_{bk}R^k_{ac}-\overline{g}_{ck}R^k_{ab}+2\overline{g}_{ak}R^k_{bc})\dot{q}^a dq^b\wedge dq^c+ \frac{1}{6}\sum_{b<c}(\overline{g}_{ck}R^k_{ab}-\overline{g}_{bk}R^k_{ac}+2\overline{g}_{ak}R^k_{cb})\dot{q}^a dq^c\wedge dq^b\\
    &=& \frac{1}{3}\sum_{b<c}(\overline{g}_{bk}R^k_{ac}-\overline{g}_{ck}R^k_{ab}+2\overline{g}_{ak}R^k_{bc})\dot{q}^a dq^b\wedge dq^c
\end{eqnarray*}
In particular we get that 
\begin{eqnarray*}
    \gamma^G&=&\frac{1}{3}\sum_{b<c}(-G_{bk}R^k_{ac}+G_{ck}R^k_{ab}-2G_{ak}R^k_{bc})\dot{q}^a dq^b\wedge dq^c.
\end{eqnarray*}
On the other hand, we have
\begin{eqnarray*}
    \Xi^G&=&-\frac{1}{2} R^k_{bc} G_{ka} \dot{q}^a dq^b\wedge dq^c= -\sum_{b<c} R^k_{bc} G_{ka} \dot{q}^a dq^b\wedge dq^c= \sum_{b<c} R^k_{cb} G_{ka} \dot{q}^a dq^b\wedge dq^c
\end{eqnarray*}

\begin{lem} 
\label{ervoor}
\begin{eqnarray*}
\gamma^G=\Xi^G & \Leftrightarrow & G_{bk}R^k_{ca}+G_{ck}R^k_{ab}+G_{ak}R_{bc}^k=0 \quad \Leftrightarrow \quad \sum_{cycl: X,Y,Z}g(X,\tilde{\mathcal Q}[Y,Z]) =0\\ &\Leftrightarrow &
 g_{bd}R^d_{ca}+g_{cd}R^d_{ab}+g_{ad}R_{bc}^d=0 \qquad
 \Leftrightarrow \quad \sum_{cycl: \tilde X,\tilde Y,\tilde Z}g^{red}(\tilde X,{\mathcal T}(\tilde Y,\tilde Z)) =0,
\end{eqnarray*}
$\forall\, X,Y,Z\in{\rm Sec}(\mathcal{D})$ and $\forall\, \tilde X, \tilde Y, \tilde Z\in\mathcal{X}(Q/G)$.
\end{lem}
\begin{proof}
The equivalences in coordinates follow from the two expressions of $\gamma^G$ and $\Xi^G$ and from the identity (\ref{useful}). For the others, remark first that the desired property in the first line is tensorial in $X,Y$ and $Z$, because for a function $f$ on $Q$:
$$
g(X,\tilde{\mathcal Q}[Y,fZ]) = f\,g(X,\tilde{\mathcal Q}[Y,Z]) + Y(f)\, g(X,\tilde{\mathcal Q}(Z)) =  f\,g(X,\tilde{\mathcal Q}[Y,Z])+0.
$$
For this reason, we only need to verify it on the frame $\{X_a\}$ of ${\rm Sec}(\mathcal D)$. Since $[X_b,X_c]=B_{bc}^kV_k = R^k_{bc}V_k$, we have $\tilde{\mathcal Q}
([X_b,X_c])=R^k_{bc}V_k$, from which the rest of the expression easily follows. For the second line, recall that $R^d_{ab}$ are the components of the gyroscopic tensor $\mathcal T$ with respect to the coordinate basis.
\end{proof}

Any of equivalent the properties in Lemma~\ref{ervoor} can now be related to $\varphi$-simplicity.
\begin{lem} \label{cyclic}
    If the Chaplygin system is $\varphi$-simple, then 
    $$
    g_{bd}R^d_{ca}+g_{cd}R^d_{ab}+g_{ad}R_{bc}^d=0.
    $$
\end{lem}
\begin{proof}
 Suppose that the $\varphi$-simplicity property (\ref{phisimpleprop}) is satisfied, that is 
 $$
 R_{ab}^d=-\frac{\partial\varphi}{\partial q^a}\delta_b^d+\frac{\partial \varphi}{\partial q^b}\delta_a^d. 
 $$
 If we add $g_{cd}$ on both sides, we get 
 $$ eq1\quad\leftrightarrow\quad g_{cd}R^d_{ab}=-\frac{\partial \varphi}{\partial q^a}g_{cb}+\frac{\partial\varphi}{\partial q^b}g_{ca}.$$
 In the above, we interchange first the indices $c$ and $b$, and then we interchange the indices $a$ and $b$. This gives the following 2 equations:
 $$eq2\quad\leftrightarrow\quad g_{bd}R^d_{ac}=-\frac{\partial \varphi}{\partial q^a}g_{bc}+\frac{\partial\varphi}{\partial q^c}g_{ba} ,$$
$$eq3\quad\leftrightarrow\quad g_{ad}R^d_{bc}=-\frac{\partial \varphi}{\partial q^b}g_{ac}+\frac{\partial\varphi}{\partial q^c}g_{ab}.$$
When we compute $eq1-eq2+eq3$, we get
$$
g_{cd}R^d_{ab}-g_{bd}R^d_{ac}+g_{ad}R^d_{bc}=0\quad\Leftrightarrow\quad g_{cd}R^d_{ab}+g_{bd}R^d_{ca}+g_{ad}R^d_{bc}=0 ,
$$
which is what we wanted to show. 
\end{proof}

This result is more interesting then it may seem at first. It tells us that, in order for a  function $\varphi$ with the $ \varphi$-simplicity property to exist,  we need to check a necessary condition for the gyroscopic tensor. From  Lemma~\ref{ervoor}, we also know that in that case the relations $d\varphi\wedge\theta_l=\Xi^G $ and $d\varphi\wedge\theta_l=\gamma^G$ will coincide. 

However, this type of cyclic obstruction does not occur in condition $(A')^G$ (or, equivalently, in $d\varphi\wedge\theta_l=\gamma^G$). If we repeat the same steps as the ones in the proof of Lemma~\ref{cyclic}, we get
$$
eq1\quad\leftrightarrow\quad \frac{\partial \varphi}{\partial q^b}g_{ac}-\frac{\partial \varphi}{\partial q^c}g_{ab}=\frac{1}{3}(\overline{g}_{bk}R^k_{ac}-\overline{g}_{ck}R^k_{ab}+2\overline{g}_{ak}R^k_{bc})  ,
$$
$$
eq2\quad\leftrightarrow\quad \frac{\partial \varphi}{\partial q^c}g_{ba}-\frac{\partial \varphi}{\partial q^a}g_{bc}=\frac{1}{3}(\overline{g}_{ck}R^k_{ba}-\overline{g}_{ak}R^k_{bc}+2\overline{g}_{bk}R^k_{ac})  ,
$$
$$
eq3\quad\leftrightarrow\quad \frac{\partial \varphi}{\partial q^a}g_{cb}-\frac{\partial \varphi}{\partial q^b}g_{ca}=\frac{1}{3}(\overline{g}_{ak}R^k_{cb}-\overline{g}_{bk}R^k_{ca}+2\overline{g}_{ck}R^k_{ba}).
$$
When we now commute $eq1+eq2+eq3$ we only get $0=0+0+0$, which means that we don't necessarily have an obstruction of this type for condition $(A')^G$.

 We can now glue all the results of this section together, we obtain  a proposition that  casts the  $\varphi$-simplicity condition as a special case of the condition $(A')^G$. 
\begin{prop} \label{prop5}
Given  $\varphi\in\mathcal{C}^\infty(Q/G)$, the following statements are equivalent:
 \begin{enumerate}
 \item[(i)] $\varphi$ satisfies the $\varphi$-simplicity property, $\quad(\Leftrightarrow\quad d\varphi\wedge \theta_l=\Xi^G)$. 
     \item[(ii)] $
         \begin{cases}
\textrm{The function $\nvarphi=\varphi$ satisfies }           (A')^G \quad(\Leftrightarrow\quad d\varphi\wedge \theta_l=\gamma^G), \\[2mm] \displaystyle \sum_{cycl: \tilde X,\tilde Y,\tilde Z}g^{red}(\tilde X,{\mathcal T}(\tilde Y,\tilde Z)) =0,\,\, \forall \tilde X,\tilde Y,\tilde Z\in\mathcal{X}(Q/G).
         \end{cases}$

 \end{enumerate}
\end{prop}
\begin{proof}
It only remains to show that $(ii) \Rightarrow (i)$. Suppose that  the condition $(A')^G$ is satisfied for $\nvarphi=\varphi$. Then, 
$$
G_{bk}R^k_{ac}+G_{ak}R^k_{bc}+g_{bc}\frac{\partial\varphi}{\partial q^a}+g_{ac}\frac{\partial\varphi}{\partial q^b}-2g_{ab}\frac{\partial\varphi}{\partial q^c}=0.
$$
When we interchange the indices $b$ and $c$, we obtain
$$
G_{ck}R^k_{ab}+G_{ak}R^k_{cb}+g_{cb}\frac{\partial\varphi}{\partial q^a}+g_{ab}\frac{\partial\varphi}{\partial q^c}-2g_{ac}\frac{\partial\varphi}{\partial q^b}=0.
$$
After subtracting these two equations from each other, we get 
$$
G_{bk}R^k_{ac}-G_{ck}R^k_{ab}+2G_{ak}R^k_{bc}-3g_{ab}\frac{\partial\varphi}{\partial q^c}+3g_{ac}\frac{\partial\varphi}{\partial q^b}=0.
$$ 
If the cyclic condition is also satisfied, then we may replace the two first terms with $G_{ak}R_{bc}^k$. Then, the equation becomes 
$$
3G_{ak}R^k_{bc}=3g_{ab}\frac{\partial\varphi}{\partial q^c}-3g_{ac}\frac{\partial\varphi}{\partial q^b}\quad\Leftrightarrow\quad g_{ad}R^d_{bc}=g_{ab}\frac{\partial\varphi}{\partial q^c}-g_{ac}\frac{\partial\varphi}{\partial q^b}, 
$$
where we have used again the property (\ref{useful}). If we change the index $a$ by $c$, index $c$ by $b$ and $b$ by $a$, we get 
$$g_{cd}R^d_{ab}=g_{ca}\frac{\partial\varphi}{\partial q^b}-g_{cb}\frac{\partial\varphi}{\partial q^a}$$
and this is the $\varphi$-simplicity condition in view of the proof of Proposition \ref{phi-simpl}. 
\end{proof}

The above proposition confirms, again, that  condition $(A')$ is more general than the $\varphi$-simplicity property.

We end this section with a coordinate-independent description of the form $\gamma$ we had introduced in coordinates. Similar to   the definition of $\Xi^G$ in (\ref{defXiG}) earlier, we may define a new 2-form $\Xi$ by  
$$
\Xi(\tilde X,\tilde Y):=\overline{g}({\rm{\bf T}}^H,{\mathcal{Q}}([{\tilde X}^H,{\tilde Y}^H])),\quad\textrm{for}\qquad \tilde X,\tilde Y\in\mathcal{X}(Q/G),
$$
where ${\mathcal{Q}}:TQ\rightarrow \mathcal{D}$. In coordinates $\Xi$ is given by 
$$
\Xi=\frac{1}{2}\overline{g}_{ai}R^i_{bc}\dot{q}^a dq^b\wedge dq^c.
$$
The difference between $\gamma$ and $\Xi$ is then 
$$
\gamma-\Xi=\frac{1}{6}(\overline{g}_{bk}R_{ac}^k-\overline{g}_{ck}R_{ab}^k-\overline{g}_{ak}R^k_{bc} )\dot{q}^a dq^b\wedge dq^c.
$$
This difference may be described globally by $\iota_\Gamma\Theta$, for any SODE $\Gamma$ on $Q/G$ (such as e.g.\ $\redvf$ is one). The 3-form $\Theta$ is such that 
$$
\Theta(\tilde X, \tilde Y, \tilde Z)=\frac{1}{9}\sum_{cycl: \tilde X,\tilde Y,\tilde Z}\overline{g}({\tilde X}^H,{\mathcal{Q}}([{\tilde Y}^H,{\tilde Z}^H])),
$$
when we identify invariant functions on $Q$ with functions on $Q/G$. 
Indeed, in coordinates we get $ \Theta=\frac{1}{18}(\overline{g}_{bk}R_{ac}^k-\overline{g}_{ck}R_{ab}^k-\overline{g}_{ak}R^k_{bc} )dq^a\wedge dq^b\wedge dq^c$. By taking the contraction with any SODE $\Gamma$ on $Q/G$, we get 
\begin{eqnarray*}
    \iota_\Gamma\Theta&=&\frac{1}{18}(\overline{g}_{bk}R_{ac}^k-\overline{g}_{ck}R_{ab}^k-\overline{g}_{ak}R^k_{bc} )\dot{q}^a dq^b\wedge dq^c -\frac{1}{18}(\overline{g}_{bk}R_{ac}^k-\overline{g}_{ck}R_{ab}^k-\overline{g}_{ak}R^k_{bc} )\dot{q}^b dq^a\wedge dq^c\\
    & & \qquad\quad +\frac{1}{18}(\overline{g}_{bk}R_{ac}^k-\overline{g}_{ck}R_{ab}^k-\overline{g}_{ak}R^k_{bc} )\dot{q}^c dq^a\wedge dq^b\\
    &= &\frac{1}{18}\dot{q}^a\Big(\overline{g}_{bk}R_{ac}^k-\overline{g}_{ck}R_{ab}^k-\overline{g}_{ak}R^k_{bc}-\overline{g}_{ak}R^k_{bc} -\overline{g}_{ck}R_{ab}^k+\overline{g}_{bk}R_{ac}^k \\&&\qquad\quad -\overline{g}_{ak}R^k_{bc} -\overline{g}_{ak}R^k_{bc}+\overline{g}_{bk}R_{ac}^k  \Big)dq^b\wedge dq^c\\
    &=&\frac{1}{6}(\overline{g}_{bk}R_{ac}^k-\overline{g}_{ck}R_{ab}^k-\overline{g}_{ak}R^k_{bc})\dot{q}^a dq^b\wedge dq^c\\
    &=& \frac{1}{6}\left[\sum_{cycl:a,b,c}\overline{g}_{bk}R^k_{ac}\right]\dot{q}^a dq^b\wedge dq^c
\end{eqnarray*}
and this is exactly the expression of the difference. Therefore, we may write $\gamma$ in a coordinate-free fashion as 
$$ 
\gamma=\Xi+\iota_\Gamma\Theta.
$$
Remark that it now also follows that  $\iota_{\Gamma}\gamma=\iota_{\Gamma}\Xi$. After substituting ${\overline g}_{ai}=-G_{ai}$ in this, we get
\begin{equation} \label{igammaisalpha}
\iota_{\Gamma}\gamma^G=\iota_{\Gamma}\Xi^G= \alphabar,
\end{equation}
and this is, in particular true for $\Gamma=\redvf$. Finally, the condition  $\displaystyle\sum_{cycl: \tilde X,\tilde Y,\tilde Z}g^{red}(\tilde X,{\mathcal T}(\tilde Y,\tilde Z)) =0$  of Proposition~\ref{prop5} is now simply $\Theta^G=0$, where $\Theta^G$ is $\Theta$ under the extra assumption of  $\ort$-orthogonality.

\section{Invariant measures} \label{sec:invmeasure}

In this section, we show that a reduced Chaplygin system may sometimes be regarded  as a Hamiltonian system with respect to a certain symplectic form. This is what we call a Hamiltonian parametrization. We will see that both the $\varphi$-simplicity property and the condition $(A')^G$ give rise to such a symplectic form, but that these two symplectic forms will, in general, not be the same, apart from the case where the cyclic condition of Lemma~\ref{ervoor} is satisfied. Moreover, we will see in this section how this question is interwoven with the existence of an invariant volume form (and therefore also an invariant measure). We will discuss both sufficient conditions and an equivalent condition for the existence of an invariant volume form. 

\subsection{Sufficient conditions}

We  show first how to construct a symplectic form from condition $(A')$. Consider the Poincar\'e 2-form $\omega_l=-d\theta_l$.

\begin{lem} \label{lem:closed}
If the couple $(\overline{g}_{bk},\nvarphi)$ satisfies $(A')$, the  2-form 
   $e^\nvarphi\tilde{\tilde{\omega}}$ (with $\tilde{\tilde{\omega}} = \omega_l-\gamma$)
   is  closed. 
\end{lem}
\begin{proof}
From Corollary~\ref{lem5}, we know that condition $(A')$ is equivalent with $d\nvarphi\wedge \theta_l=\gamma$.  We need to show that $e^\nvarphi(\omega_l-\gamma)$ is closed. This is the case:
$$
d(e^\nvarphi(\omega_l-\gamma))= d(e^\nvarphi\omega_l-e^\nvarphi d\nvarphi\wedge\theta_l) = e^\nvarphi d\nvarphi\wedge \omega_l-e^\nvarphi d\nvarphi\wedge d\nvarphi\wedge\theta_l + e^\nvarphi d\nvarphi \wedge d\theta_l
= 0.
$$
\end{proof}

Recall that ${\rm{dim}}(Q/G) =m$ and ${\rm{dim}}(T(Q/G)) =2m$. A $2m$-form  is therefore a top form on $T(Q/G)$.  

In the special case of  $\ort$-orthogonality (where $\overline{g}_{ai}=-G_{ai}$), we will denote the 2-form $\omega_l-\gamma^G$ as $\tilde{\tilde{\omega}}^G$. We may then state the following result.

\begin{prop} \label{invvolformG}
If $\nvarphi\in\mathcal{C}^\infty(Q/G)$ satisfies $(A')^G$ (i.e.\ when  the couple $(\overline{g}_{ai}=-G_{ai},\nphi=\nvarphi\circ\pi)$ is a projective geodesic extension), then 
     $$
     e^{(m-1)\nvarphi}(\tilde{\tilde{\omega}}^G)^m =      e^{(m-1)\nvarphi}\omega_l^m 
     $$ 
     is an invariant volume form. 
 \end{prop}
\begin{proof}
From Lemma~\ref{lem:closed}, we know that the 2-form $\tilde{\tilde{\omega}}^G$ is closed. With the help of the property (\ref{igammaisalpha}),  the equation $\iota_{\redvf}{\omega}_l=dE_l+\alphabar$ can be rewritten as $\iota_{\redvf}({\omega}_l-\gamma^G)=dE_l$.

Given that also  $\iota_{e^{-\nvarphi}\redvf}(e^{\nvarphi}\tilde{\tilde{\omega}}^G)=dE_l$, we see  that $e^{-\nvarphi}\redvf$ is a Hamiltonian vector field with respect to  the closed form $e^\nvarphi \tilde{\tilde{\omega}}^G$. We can now compute the Lie derivative of $e^{\nvarphi}\tilde{\tilde{\omega}}^G$ with respect to this vector field $e^{-\nvarphi}\redvf$: 
$$
\mathcal{L}_{e^{-\nvarphi}\redvf}(e^{\nvarphi}\tilde{\tilde{\omega}}^G)=\iota_{e^{-\nvarphi}\redvf}(d(e^{\nvarphi}\tilde{\tilde{\omega}}^G))+d(\iota_{e^{-\nvarphi}\redvf}(e^{\nvarphi}\tilde{\tilde{\omega}}^G))=0+d(dE_l)=0.
$$
From this, we also get that (because of the involved top-forms)
\begin{eqnarray*}
0&=&\mathcal{L}_{e^{-\nvarphi}\redvf}((e^{\nvarphi}\tilde{\tilde{\omega}}^G)^m) = \iota_{e^{-\nvarphi}\redvf} \left(d((e^{\nvarphi}\tilde{\tilde{\omega}}^G)^m)\right) + d \left(\iota_{e^{-\nvarphi}\redvf} ((e^{\nvarphi}\tilde{\tilde{\omega}}^G)^m)\right) \\&=& 0 + d \left(\iota_{\redvf} (e^{(m-1)\nvarphi} (\tilde{\tilde{\omega}}^G)^m)\right)= \mathcal{L}_{\redvf} \left(e^{(m-1)\nvarphi}(\tilde{\tilde{\omega}}^G)^m\right) - \iota_{\redvf} \left( d (e^{(m-1)\nvarphi} (\tilde{\tilde{\omega}}^G)^m) \right) \\& = &\mathcal{L}_{\redvf}(e^{(m-1)\nvarphi}(\tilde{\tilde{\omega}}^G)^m) -0,
\end{eqnarray*}
meaning  that $e^{(m-1)\nvarphi}(\tilde{\tilde{\omega}}^G)^m$ is an invariant volume form. This  part of the proof is, in fact, an application of a more general property that is to be found in Proposition~3.18 of \cite{Garcia}, and attributed there to e.g.\ \cite{Ehlers,FJ}. 

Finally, it is easy to see that the top forms $(\tilde{\tilde{\omega}}^G)^m$ and $\omega_l^m$ coincide, because of the fact that $\gamma^G$ is semi-basic.
\end{proof}

We are now ready to explain a result of e.g.\ \cite{Garcia}, the help of the above Propositions, as a special case: One of the sufficient conditions for the existence of an invariant volume form on the reduced space is the $\varphi$-simplicity property. 

\begin{cor} If the system is $\varphi$-simple, the reduced Chaplygin equations have an invariant volume form $e^{(m-1)\varphi}\tilde{\omega}^m= e^{(m-1)\varphi}\omega_l^m$ (where $\tilde{\omega}=\omega_l-\Xi^G$). \label{invvolform}
\end{cor}
\begin{proof}
From Proposition~\ref{prop5} we know that, if the system is $\varphi$-simple, the function $\nvarphi=\varphi$ satisfies $(A')^G$. We also know from Lemma~\ref{cyclic} that the cyclic condition is satisfied, from which $\gamma^G =\Xi^G$ (Lemma~\ref{ervoor}). Therefore, we are in the situation of Proposition~\ref{invvolformG}, and $\tilde{\tilde{\omega}}^G = \tilde\omega$.
\end{proof}
Even though the notion of $\varphi$-simplicity does not appear in \cite{Cantrijn, ohsawa} as such, a version of the above result can also be found there, if one takes the equivalent characterization of Proposition~\ref{phi-simpl} into account.

Notice that even though the symplectic form $\omega_l-\Xi^G$ is different from the one used in the previous proposition, namely $ \omega_l-\gamma^G$ (since, in general, $\gamma^G\neq \Xi^G$), the volume forms will still be the same, $e^{(m-1)\varphi}\omega_l^n$.

We conclude that $\varphi$-simplicity is not a necessary condition for the existence of an invariant volume form. The subtlety lies in the following observations.
 The reduced Chaplygin system is given by 
$$
\iota_{\redvf}\omega_l=d E_l+\alphabar.
$$
In \cite{ohsawa, Cantrijn, Garcia} the gyroscopic 1-form $\alphabar$ is defined in such a way that  
$ \alpha=\iota_{\redvf}\Xi^G$. This is then used to write the reduced system as 
$$
\iota_{\redvf}\tilde{\omega}=d E_l, 
$$
where $\tilde{\omega}=\omega_l-\Xi^G$. When the system is $\varphi$-simple, $\tilde\omega$ is symplectic, and the reduced system is Hamiltonian.

However, in the proof of Proposition~\ref{invvolformG}, we have seen that $\Xi^G$ is not the only 2-form that satisfies $\iota_{\redvf}\Xi^G=\alphabar$. Indeed, the 2-form $\gamma^G$ also satisfies $\iota_{\redvf}\gamma^G=\alphabar$ and we have seen that the reduced Chaplygin system  can also be written as 
$$
\iota_{\redvf}\tilde{\tilde{\omega}}^G=d E_l,
$$
where $ \tilde{\tilde{\omega}}^G=\omega_l-\gamma^G$. Lemma~\ref{lem:closed} says that this will be a Hamiltonian system when condition $(A')^G$ is satisfied. 

These results  fit in the framework of e.g.\ \cite{ BalseiroFernandez, Balseiro, Sansonetto,Balseiro2,Balseiro3} where, roughly speaking, a certain semi-basic two-form $B$ is added to the 2-form $\omega_l-\Xi^G$ in order to write the reduced system as
$$
\iota_{\redvf}(\omega_l-\Xi^G-B)=dE_l.
$$
Among other conditions, this 2-form $B$ needs to satisfy $\iota_{\redvf}B=0$. If that is the case,  $B$ is called a {\em dynamical gauge transformation} in \cite{Balseiro}. In our setting, the difference between $\gamma^G$ and $\Xi^G$ is such a 2-form $B$, since $\iota_{\redvf}\Xi^G= \alphabar=\iota_{\redvf}\gamma^G$. In fact, in the previous section, we saw that the components of this difference $B$ are actually the cyclic coefficients that we encountered in e.g.\ Lemma~\ref{ervoor}.

\subsection{Equivalent condition}

In \cite{Cantrijn} one may find an equivalent condition for the existence of an invariant volume form. This condition makes use of a certain 1-form $\beta$ on the reduced space $Q/G$. It is defined as the following contraction with the gyroscopic tensor $\mathcal{T}$:
\begin{equation}\label{contraction}
\beta(\tilde Y):=\sum_a\left<Z^a,\mathcal{T}\left(Z_a,\tilde Y\right)\right>,\quad\forall \tilde Y\in\mathcal{X}(Q/G),
\end{equation}
where $\{Z_a\}$ is a basis for $\mathcal{X}(Q/G)$ and $\{Z^a\}$ is its dual. In local coordinates $\beta$ is given by 
$$
\beta=\sum_{a,b}R_{ab}^a dq^b.
$$
We refer to \cite{Cantrijn} for the proof of the following statement (and to \cite{Garcia} for an alternative proof):

\begin{prop} \label{propCantrijn}
 The reduced equations of motion of a Chaplygin system possess a basic invariant volume form  if and only if $\beta$ is exact (locally: if and only if $d\beta=0$).  
\end{prop}

The 1-form $\beta$ is exact if and only if there exists a function $S$ on $Q/G$ such that 
\begin{equation} \label{equivcond}
\frac{\partial S}{\partial q^e}=\sum_{a}R_{ae}^a,
\end{equation}
and, locally, this is the case if and only if 
$$
\frac{\partial R_{ae}^a}{\partial q^f}=\frac{\partial R_{af}^a}{\partial q^e}.
$$

We want to see what the obstruction is for the equivalence between $(A')^G$ and the existence of an invariant volume form. First, we already know from Proposition~\ref{invvolformG}  that one direction is always satisfied. We give an alternative proof for this property. If  condition $(A')^G$ holds for $\nvarphi$, we have seen that 
\begin{eqnarray*}
    (A')^G &\Leftrightarrow& d\nvarphi\wedge \theta_l=\gamma^G\\
    &\Leftrightarrow& g^{cd}( -G_{bc}R_{ac}^e-G_{ae}R_{bc}^e-g_{bc}\frac{\partial\nvarphi}{\partial q^a}-g_{ac}\frac{\partial \nvarphi}{\partial q^b}+2g_{ab}\frac{\partial \nvarphi}{\partial q^c})=0
\end{eqnarray*}
Here, the indices $d$ and $b$ are free. When we take $d=b$ and sum over all $b$, we get 
\begin{eqnarray*}
   (A')^G &\Rightarrow& -g^{cb}g_{be}R^e_{ac}-g^{cb}g_{ae}R^e_{bc}-m\frac{\partial\nvarphi }{\partial q^a}-\frac{\partial \nvarphi}{\partial q^a}+2\delta_a^c\frac{\partial\nvarphi}{\partial q^c}=0\\
   &\Leftrightarrow& -R_{ac}^c-0-(m-1)\frac{\partial\nvarphi}{\partial q^a}=0.
   \\
   &\Leftrightarrow& \frac{\partial(m-1)\nvarphi}{\partial q^a}=R_{ca}^c
\end{eqnarray*}
(with sum over $c$). If we take $S=(m-1)\nvarphi$,  then we recover  the equivalent condition (\ref{equivcond}) for the existence of an invariant volume form. Based on Proposition~\ref{invvolformG}, the invariant volume form is of the form $e^{(m-1)\nvarphi}\omega_l^n=e^{S}\omega_l^n$.

The obstruction to equivalence will appear in the other direction. Suppose that $\nvarphi$ already satisfies 
$$
d((m-1)\nvarphi)=\beta.
$$
This $\nvarphi$ will satisfy also $(A')$, whenever 
$$ 
-G_{be}R_{ac}^e-G_{ae}R_{bc}^e-g_{bc}\left(\frac{1}{m-1}R_{ea}^e\right)-g_{ac}\left( \frac{1}{m-1}R_{eb}^e\right)+2g_{ab}\left(\frac{1}{m-1}R_{ec}^e \right)=0.
$$
Similarly as in the proof of Corollary~\ref{lem5}, this expression can be seen to be equivalent with the property
$$
\frac{\beta}{m-1}\wedge\theta_l=\gamma^G.
$$
This is the sought obstruction for an invariant volume form to give a projective geodesic extension. We may conclude: 
\begin{prop} \label{propprojgeodext} Let $\beta$ be the contraction (\ref{contraction}) of the Chaplygin system. The following statements are equivalent:
\begin{enumerate} \item There exists a locally defined  $\nvarphi$ that satisfies $(A')^G$ (i.e.\ there exist a projective geodesic extension $(\overline{g}_{ai}=-G_{ai},\nphi=\nvarphi\circ\pi)$).
\item   $\beta$ satisfies $\begin{cases}
        d\beta=0,\\[1mm]
\displaystyle        \frac{\beta}{m-1}\wedge\theta_l=\gamma^G.
    \end{cases}$
    \end{enumerate}
\end{prop}

We can also relate this back to $\varphi-$simplicity. Again, one direction is trivial:
\begin{eqnarray*}
    \varphi-\textrm{simplicity}&\Leftrightarrow& d\varphi\wedge\theta_l=\Xi^G\\[1mm]
    &\Leftrightarrow& \frac{\partial\varphi}{\partial q^b}\delta_a^d-\frac{\partial\varphi}{\partial q^a}\delta_b^d=R_{ab}^d.
\end{eqnarray*}
If we take $b=d$ and sum over $b$ we get $\displaystyle\frac{\partial(m-1)\varphi}{\partial q^a}=R_{ba}^b$, and therefore the equivalent condition for an invariant volume form.

Based on  Corollary~\ref{invvolform}, this invariant volume form is $e^{(m-1)\varphi}\omega_l^m$, where now $\varphi$ is defined by the $\varphi$-simplicity condition.

Notice that even though the symplectic form $\omega_l-\Xi^G $ is different from the one used in the previous section, namely $ \omega_l-\gamma^G$ (in general $\gamma^G\neq \Xi^G$), the volume forms will still be the same since also here one can show that $$
e^{(m-1)\varphi}(\omega_l-\Xi^G)^n=e^{(m-1)\varphi}\omega_l^m .
$$

Conversely, suppose that there exists a function $\varphi$ such that $d(m-1)\varphi=\beta$. This $\varphi$ will satisfy $\varphi$-simplicity when 
$$
\frac{1}{n-1}R_{eb}^e\delta_a^d-\frac{1}{n-1}R_{ea}^e\delta_b^d=R_{ab}^d \quad\Leftrightarrow\quad \frac{\beta}{m-1}\wedge\theta_l=\Xi^G.
$$ 

Therefore, we may conclude that 
\begin{prop} \label{propphisimple}
Let $\beta$ be the contraction (\ref{contraction}) of the Chaplygin system. The following statements are equivalent: 
\begin{enumerate} \item There exists a locally defined $\varphi \in {\mathcal C}^\infty(Q/G)$ for which the system is $\varphi$-simple.
 \item $\beta$ satisfies $\begin{cases}
          d\beta=0, \\[2mm]
\displaystyle          \frac{\beta}{m-1}\wedge\theta_l=\Xi^G.
         \end{cases}$ $\Leftrightarrow$ $\begin{cases}
          d\beta=0, \\[2mm]
          \displaystyle\frac{\beta}{m-1}\wedge\theta_l=\gamma^G,\\[2mm]\Theta^G=0.
         \end{cases}$
\end{enumerate}
\end{prop}

\begin{cor}
When the dimension of $Q/G$ is 1 or 2,  the statements in Proposition~\ref{propCantrijn}, Proposition~\ref{propprojgeodext} and Proposition~\ref{propphisimple}  are all equivalent with $d\beta=0$. 
\end{cor}
\begin{proof}
When the dimension of $Q/G$ is 2, there are no 3-forms on $Q/G$, and therefore the cyclic condition is void and $\Theta^G=0$. Also, $\beta = -R^2_{12}dq^1+R^1_{12}dq^2$, $\theta_l = (g_{11}{\dot q}^1+g_{12}{\dot q}^2)dq^1 + (g_{12}{\dot q}^1+g_{22}{\dot q}^2)dq^2$, $\gamma^G=\Xi^G = -(R^1_{12} ((g_{11}{\dot q}^1+g_{12}{\dot q}^2)) + R^2_{12} (g_{12}{\dot q}^1+g_{22}{\dot q}^2)) dq^1\wedge dq^2$, and therefore automatically $\beta\wedge \theta_l \equiv \gamma^G$. So, the only condition that remains is  $d\beta=0$. 
\end{proof}

In 2 dimensions, the sole condition $d\beta =0$ is 
$$
\frac{\partial R^2_{12}}{\partial q^2} + \frac{\partial R^1_{12}}{\partial q^1} =0.
$$
This necessary and sufficient condition, for either one of the above properties, is well-known (via  many different approaches), and its relation to the invariant volume form is often called `Chaplygin reducing multiplier Theorem' (see e.g.\ \cite{oscartom,Garcia2,Garcia} or the translation of the original paper  \cite{Chaplygin2} that inspired this topic in the literature).

\subsection{The generalized  nonholonomic particle}

We consider again the Chaplygin system of the generalized nonholonomic particle (where $\rho(x,y)$), but now only under the assumption that $\overline{g}_{ai}=-G_{ai}$. Here this means that
$$
{\overline g}_{xz}=-\rho(x,y)  ,\qquad {\overline g}_{yz}=0.
$$

 Since the dimension of $Q/G$ is 2, we know that all the different notions of the previous sections, \{$\varphi$-simplicity, projective geodesic extension and invariant volume form\}, coincide. We start again from scratch and we see what information we can obtain from the remaining condition $d\beta=0$. 

Since $[X_x,X_y]=-\frac{\partial\rho}{\partial y}X_z-\frac{\rho\frac{\partial\rho}{\partial y}}{1+\rho^2}X_x$, we have that 
$$ R_{xy}^x=-\frac{\rho\frac{\partial\rho}{\partial y}}{1+\rho^2}\quad\textrm{and}\quad R_{xy}^y=0.$$
Condition $d\beta=0$ holds if and only if $\displaystyle\frac{\partial R^x_{xy}}{\partial x}= -  \frac{\partial R^y_{xy}}{\partial y} 
$ and in this case this is 
$$
\frac{\partial}{\partial x}\left( -\frac{\rho\frac{\partial\rho}{\partial y}}{1+\rho^2}\right)=0.
$$ 
This means that $$
-\frac{\rho\frac{\partial\rho}{\partial y}}{1+\rho^2}=k
$$
for some function $k$ that only depends on $y$. Let $K$ be a primitive function of $k$, i.e. $k(y) =K'(y)$. When we integrate both sides we get that 
$$
-\frac{1}{2}\ln(1+\rho(x,y)^2)=K(y)+L(x),
$$
for some arbitrary function $L$ that only depends on $x$. From this, to get a projective geodesic extension,  the function $\rho(x,y)$ must be of the form
$$ 
\rho(x,y)=\pm\sqrt{e^{2(K(y)+L(x))}-1}.
$$

The corresponding conformal factor  $\nvarphi$ is such that $\frac{\partial (n-1)\nvarphi}{\partial q^a}=R_{ca}^c$. That is
\begin{equation}
    \begin{cases}
        \frac{\partial\nvarphi}{\partial x}=0\\
        \frac{\partial\nvarphi}{\partial y}=-\frac{\rho\frac{\partial\rho}{\partial y}}{1+\rho^2}=K'
    \end{cases}
\end{equation}
From these PDEs we can conclude that $\nvarphi$ is a function of $y$ only, and that
$$
\nvarphi(y)= K(y)=-\frac{1}{2}\ln(1+\rho^2(x,y))-L(x),
$$
where the ostensible dependence on $x$ in the right hand side cancels out.

This is consistent with the earlier results of Section~\ref{firstmodpar} and Section~\ref{secondmodpar}. The case where $\rho$ only depends on $y$ is the one where $G(x)=G_0$ is constant. Then $
\nvarphi(y)= -\frac{1}{2}\ln(1+\rho^2(y)),
$ is (up to  a constant) the conformal factor of the projective geodesic extension (keeping in mind that, here, we assume  $\overline{g}_{ai}=-G_{ai}$).

\subsection{A mathematical example} \label{mathex}

Based on the results in Propositions~\ref{propCantrijn}, \ref{propprojgeodext} and \ref{propphisimple} we know that $(A')^G$ is a weaker condition than the $\varphi$-simplicity property, but a stronger condition than having only an invariant volume form.  In \cite{Jovanovic} it is shown that the ball rolling over a sphere is an example of a Chaplygin system that possesses an invariant measure, but that is not Hamiltonizable, and thus also not $\varphi$-simple. To  further verify that the item $(ii)$ of the classification is not void, we give a (mathematical) example of a nonholonomic system that is not $\varphi$-simple, but that does have a solution $\nvarphi$ of condition $(A')^G$. Such a system will allow for a projective geodesic extension (with $\overline{g}_{ai}=-G_{ai}$) and an invariant volume form,  without  being $\varphi$-simple. 

Consider on   ${\mathbb R}^4$ (with coordinates $(x,y,z,u)$) the nonholonomic  system with the somewhat overwhelming Lagrangian function 
\begin{eqnarray*}
 L
 &=&\frac{1}{2}\Big((1+2(y-z)+4(y-z)^2)\dot{x}^2+(1+2(z-x)+4(z-x)^2)\dot{y}^2+(1+2(x-y)+4(x-y)^2)\dot{z}^2\\
 & &+2((x-z)+4(x-y)(y-z))\dot{x}\dot{z}+2((y-x)+4(z-x)(y-z))\dot{x}\dot{y} \\
 & & +2((z-y)+4(z-x)(x-y))\dot{y}\dot{z} +2(1+4(y-z))\dot{x}\dot{u}+2(1+4(z-x))\dot{y}\dot{u} \\
 & & +2(1+4(x-y))\dot{z}\dot{u}+4\dot{u}^2\Big)   
\end{eqnarray*}
and a  nonholonomic constraint given by 
$$
\dot{u}=-(y-z)\dot{x}-(z-x)\dot{y}-(x-y)\dot{z}.
$$
The distribution $\mathcal{D}$ that describes this constraint is spanned by the vector fields
$$
X_x=\frac{\partial}{\partial x}-(y-z)\frac{\partial}{\partial u},\quad X_y=\frac{\partial}{\partial y}-(z-x)\frac{\partial}{\partial u},\quad X_z=\frac{\partial}{\partial z}-(x-y)\frac{\partial}{\partial u}.
$$
Both the distribution and the Lagrangian function are invariant with respect to the Lie group action of $G=\mathbb{R}$ on ${\mathbb R}^4$, given by 
$$ 
G\times Q\rightarrow Q,(a,(x,y,z,u))\mapsto (x,y,z,u+ a).
$$
The vertical distribution with respect to the principal fiber bundle $\pi:Q\rightarrow Q/G$ is then 
$$
V\pi=\textrm{span}\left\{V_u:=\frac{\partial}{\partial u}\right\}.
$$
The nonholonomic system is Chaplygin because  the distribution $\mathcal{D}$ may be seen as the horizontal distribution of a principal connection. 
With this, the tangent bundle $TQ$ may be decomposed as
$$
TQ=\mathcal{D}\oplus V\pi.
$$
When we rewrite the Lagrangian in terms of the quasi-velocities $(\tilde{v}^a,\tilde{v}^u)$ of the basis $\{X_a,V_u\}$, we get the more compact expression  
\begin{eqnarray*}
L&=&\frac{1}{2}g(\tilde{v},\tilde{v})=\frac{1}{2}( G_{ab}\tilde{v}^a \tilde{v}^b+2G_{au}\tilde{v}^a\tilde{v}^u+G_{uu}(\tilde{v}^u)^2)\\
&=&\frac{1}{2}({\tilde{v}_x}^2+{\tilde{v}_y}^2+{\tilde{v}_z}^2+2{\tilde{v}_x}{\tilde{v}_u}+2{\tilde{v}_y}{\tilde{v}_u}+2{\tilde{v}_z}{\tilde{v}_u}+4{\tilde{v}_u}^2) ,   
\end{eqnarray*}
where $G_{au}=g(X_a,\frac{\partial}{\partial u})$ and $G_{ab}=g_{ab}=g(X_a,X_b)$. For what follows, it is important to note that $G_{au}=1$ (for all $a$'s). 

From the current frame, we can  create an orthogonal basis $\{X_u\}$ for $\mathcal{D}^g$, where 
$$
X_u=\frac{\partial}{\partial u}-G_{au}G^{ab}X_b=\frac{\partial}{\partial u}-X_x-X_y-X_z=\frac{\partial}{\partial u}-\frac{\partial}{\partial x}-\frac{\partial}{\partial y}-\frac{\partial}{\partial z}.
$$
We denote the quasi-velocities with respect to this basis $\{X_a,X_u\}$ by $(v^a,v^u)$ and the corresponding  coefficients of the kinetic energy metric as $g_{ab}=g(X_a,X_b), g_{au}=g(X_a,X_u)$ and $g_{uu}=g(X_u,X_u)$. 

The goal, here, is to investigate both the $\varphi$-simplicity property and the condition $(A')$. For this we need the bracket coefficients $R_{ab}^u$ and $R_{ab}^c$ with respect to the orthogonal basis $\{X_a,X_u\}$: 
$$
[X_x,X_y]=2\frac{\partial}{\partial u}=2(X_u+X_x+X_y+X_z),\quad [X_x,X_z]=-2\frac{\partial}{\partial u}=-2(X_u+X_x+X_y+X_z),
$$ 
$$
[X_y,X_z]=2\frac{\partial}{\partial u}=2(X_u+X_x+X_y+X_z).
$$
From this, we deduce that $R_{xy}^u=2,R_{xz}^u=-2$ and $R_{yz}^u=2$ and $R_{xy}^a=2,R_{xz}^a=-2$ and $R_{yz}^a=2$ (for all $a$'s).

If the system is $\varphi$-simple, then we should be able to find a solution $\varphi$ of the system of partial differential equations given by 
$$ 
\frac{\partial \varphi}{\partial q^b}\delta_c^d-\frac{\partial\varphi}{\partial q^c }\delta_b^d=R_{cb}^d.
$$
When $b=x,c=y$ we get for $d=x,y,z$ the following three equations 
$$ \frac{\partial \varphi}{\partial x}=R_{xy}^x=2,\quad \frac{\partial \varphi}{\partial y}=2,\quad 0=0.$$
When we take $b=z,c=x$ we get for $d=x,y,z$
$$ 
\frac{\partial\varphi}{\partial z}=R_{xz}^x=-2,\quad 0=0,\quad \frac{\partial \varphi}{\partial x}=-2.
$$
Since this is an inconsistency, there is no such $\varphi$. In fact, it is not a surprise that the system is not $\varphi$-simple, since it also does not satisfy the required cyclic condition: 
$$
G_{bu}R_{ca}^u+G_{cu}R_{ab}^u+G_{au}R_{bc}^u=0.
$$
Indeed, if we take for example $ a=x,b=y$ and $c=z$, we get that 
$$
G_{yu}R_{zx}^u+G_{zu}R_{xy}^u+G_{xu}R_{yz}^u=2+2+2=6\neq 0.
$$

Now, we look at condition $(A')^G$, that is: condition $(A')$, where we substitute  $\overline{g}_{au}=-G_{au}=-1$. The condition becomes here
$$ 
-G_{bu}R_{ac}^u-G_{au}R^u_{bc}-g_{bc}\frac{\partial\nvarphi}{\partial q^a}-g_{ac}\frac{\partial\nvarphi}{\partial q^b}+2g_{ab}\frac{\partial\nvarphi}{\partial q^c}  = 0.
$$
For $a=x,b=y,c=z$ we get the differential equation $$ \frac{\partial\nvarphi}{\partial x}+\frac{\partial\nvarphi}{\partial y}=2\frac{\partial \nvarphi}{\partial z}.
$$
For $a=z,b=y,c=x$ we have $$ \frac{\partial\nvarphi}{\partial z}+\frac{\partial\nvarphi}{\partial y}=2\frac{\partial\nvarphi}{\partial x}.
$$
When $a=x,b=z,c=y$, the equation becomes 
$$ \frac{\partial\nvarphi}{\partial x}+\frac{\partial\nvarphi}{\partial z}=2\frac{\partial\nvarphi}{\partial y}.
$$
Other combinations give the same equations as the ones above. This system of partial differential equations has a solution  
$$ 
\nvarphi(x,y,z)=h(x+y+z),
$$
where $h$ is any smooth function of one variable.

We conclude that, despite the fact that the system is not $\varphi$-simple, we do have a projective geodesic extension that is a $\mathcal{D}$-conformal modification of $g$. Indeed, we have that  $(\nphi=\nvarphi\circ\pi,\overline{g}_{ai}=-G_{ai}=-1)$, for the $\nvarphi$ above, is such a projective geodesic extension. The nonholonomic trajectories are pregeodesics of the metric defined by the  Lagrangian
\begin{eqnarray*}
\hat{L}&=&\frac{1}{2}\hat{g}(v,v) = \frac{1}{2}e^{2\nvarphi}\overline{g}(v,v)=\frac{1}{2}e^{2\nvarphi}(\overline{g}_{ab}v^a v^b+2\overline{g}_{au}v^av^u+\overline{g}_{uu}(v^u)^2) \\
&=&\frac{1}{2}e^{2\nvarphi}(G_{ab}v^a v^b-2G_{au}v^av^u+\overline{g}_{uu}(v^u)^2)\\ 
&=&\frac{1}{2}e^{2h(x+y+z)}(v_x^2+v_y^2+v_z^2-2{v_x}{v_u}-2{v_y}{v_u}-2{v_z}{v_u}+\overline{g}_{uu}v_u^2),    
\end{eqnarray*}
where $\hat L$ is given here in the coordinates $(x,y,z,u,v_x,v_y,v_z,v_u)$  with respect to the basis $\{X_a,X_u\}$. We may choose $\overline{g}_{uu}=\alpha\in\mathbb{R}$  in such a way that this metric $\hat{g}$ is positive definite. We need to choose the coefficients $\overline{g}_{uu}$ such that the matrix 
$$\frac{1}{2}e^{2h(x+y+z)}\begin{pmatrix}
    1&0&0&-1\\0&1&0&-1\\0&0&1&-1\\-1&-1&-1&\alpha
\end{pmatrix}$$
is positive definite.  If we set $a:=\frac{1}{2}e^{2h(x+y+z)}$, then the characteristic polynomial of this matrix is given by $(a-\lambda)^2(a^2\alpha-3a^2+(-a\alpha-a)\lambda+\lambda^2)=0 $.  This means that there exists a (strictly positive)  eigenvalue $\lambda=a$, with multiplicity 2. The two other eigenvalues have a sum of $a(\alpha+1)$ and product $a^2(\alpha-3)$. We can ensure that they are both strictly larger than zero by taking e.g.\ $\alpha=\overline{g}_{uu}=42$. 

\section{Hamiltonian reparametrization by geodesic equations} 

So far, we have only discussed projective geodesics extensions. They have the property that the unreduced nonholonomic trajectories can be regarded as reparametrized geodesics of a Riemannian metric. Now, we look at the solutions of the reduced equations. We have said in the Introduction that, when they  can be rewritten as 
Hamiltonian equations with respect to some symplectic or Poisson structure, we call this a Hamiltonian reparametrization.

In \cite{Garcia}, it is shown that, when the system is $\varphi$-simple, the solutions of the reduced Chaplygin equations give such a Hamiltonian reparametrization. In \cite{Simoes} it is further shown that in the  purely kinetic case, these reduced equations can be regarded, themselves, as pregeodesics of a metric on the reduced space $Q/G$. They call this metric the {\em canonical metric} $g^{can}$. If the Lagrangian function is given  by $L(v)=\frac{1}{2}g(v,v)$, then the expression of this metric on $Q/G$ is 
$$
g^{can}=e^{2\varphi}{g}^{red},
$$
where $\varphi\in\mathcal{C}^\infty(Q/G)$ is the function that satisfies the $\varphi$-simplicity condition and ${g}^{red}$ is the reduced metric (\ref{gred}) that we had already encountered in Section~\ref{prelimChap}.

The authors of \cite{Simoes} then use this to create (what we now call) a projective geodesic extension $\Gamma^{\h}$ of the kinetic nonholonomic system $\Gammanh$ (after reparametrization). The vector field $\Gamma^{\h}$ is the geodesic spray of the metric on $Q$ that is defined from the canonical metric $g^{can}$ on $Q/G$ by:
\begin{enumerate}
    \item $h(X_a,X_b)=g^{can}\displaystyle\left(\frac{\partial}{\partial q^a},\frac{\partial}{\partial q^b}\right)$ for all $X_a,X_b\in{\rm Sec}(\mathcal{D}$);
    \item $h(V_i,V_j)=g(V_i,V_j)$ for all $V_i,V_j\in{\rm Sec}(V\pi)$;
    \item $h(X_a,V_k)=0$ for all $X_a\in{\rm Sec}(\mathcal{D}),V_k\in {\rm Sec}(V\pi)$. 
\end{enumerate}
Since here specifically $h_{ij}=G_{ij}$, we already know from Proposition~\ref{prop1} that, thanks to the freedom in ${\hat g}_{ij}$,  this is actually only one member of a whole class of $\ort$-orthogonal projective geodesic extensions.
 
This is not the approach we have taken so far. In the first part of this paper we considered a more general class of nonholonomic systems (more general than Chaplygin systems). Since they do not necessarily have a symmetry group, our main goal was to obtain a (projective) geodesic extension without having to reduce the system. Indeed, conditions $(A')$ and $(B')$ result in a geodesic extension (after reparametrization) on the original configuration space $Q$, without using the information of the reduced dynamics. Nevertheless, once we have a projective geodesic extension on the full space $Q$, one may wonder whether it also leads to a representation of the reduced equations as geodesic equations. In this section, we show that, again, $\varphi$-simplicity is not a necessary condition for this to happen, if we extend our point of view  from only the case where $\overline{g}_{ai}=-G_{ai}$.  

Before continuing, we remark first that, since the reduced equations are a full system of second-order differential equations on $Q/G$, when we reinterpret  them as (pre)geodesics, we can no longer speak of an `extension', since there will be a one-to-one identification between the solutions of the reduced equations and the geodesics.  

\begin{prop} \label{prop14}     Suppose that we have a projective geodesic extension  $(\overline{g}_{ai}, \nphi=\nvarphi\circ\pi)$ for a Chaplygin system.
   
\begin{enumerate} \item[(i)] If $\gamma$ satisfies 
    $$
\iota_{\redvf}\gamma=\alphabar,
$$
(or, equivalently, if $B=\gamma-\Xi^G$ is a dynamical gauge transformation),
    then the projective geodesic extension leads to a Hamiltonian reparametrization: the reduced trajectories are (a complete set of) pregeodesics of a Riemannian metric.

    \item[(ii)]
    If the projective geodesic extension is  $\ort$-orthogonal (of the form $(\overline{g}_{ai}=-G_{ai}, \nphi=\nvarphi\circ\pi)$), we always have a Hamiltonian reparametrization.  
 \end{enumerate}
\end{prop}

\begin{proof}
 (i) When conditions $(A')$ and $(B')$ are satisfied for the couple $(\overline{g}_{ai}, \nphi=\nvarphi\circ\pi)$, then $d\nvarphi\wedge\theta_l=\gamma$. Moreover, we assume that 
$
\alphabar=\iota_{\redvf}\gamma=\redvf(\nvarphi)\theta_l -2ld\nvarphi
$, 
which means that the reduced vector field can be determined from the equation $\iota_{\redvf}(\omega_l-\gamma)=dE_l$, or, by Lemma~\ref{lem:closed} that $e^{-\nvarphi}\redvf$ is a Hamiltonian vector field with respect to $e^{\nvarphi} \tilde{\tilde\omega}$.
 
 For the Poincar\'e forms $\theta_l=g_{ab}\dot{q}^adq^b$ and $\omega_l=-d\theta_l=-d(g_{ab}\dot{q}^a)\wedge dq^b$, we always have $\iota_{\Delta}\theta_l=0$ and $\iota_{\Delta}\omega_l =-\iota_{\Delta}d(g_{ab}\dot{q}^a)=-g_{ab}\dot{q}^adq^b=-\theta_l$, where $\Delta={\dot q}^a \displaystyle \frac{\partial}{\partial {\dot q}^a}$ is the Liouville vector field. Also $\iota_\Delta\alphabar=0$.

Let now $\Gamma^\h$ be the geodesic spray on $T(Q/G)$ of the metric $\h=(\h_{ab})=(e^{2\nvarphi}g_{ab})$ on $Q/G$. This vector field is characterized by the symplectic equation 
$$ 
\iota_{\Gamma^\h}\omega_\h=dE_\h,
$$
where $E_\h=e^{2\nvarphi}g_{ab}\dot{q}^a\dot{q}^b=e^{2\nvarphi}E_l=e^{2\nvarphi}l$. Since $\theta_\h= e^{2\nvarphi}g_{ab}\dot{q}^adq^b=e^{2\nvarphi}\theta_l$, the 2-form $\omega_\h=-d\theta_\h$ is 
\begin{eqnarray*}
    \omega_\h&=&-d(e^{2\nvarphi} g_{ab}\dot{q}^a)\wedge dq^b= e^{2\nvarphi}\omega_l-2 g_{ab}\dot{q}^a e^{2\nvarphi}d\nvarphi \wedge dq^b= e^{2\nvarphi}\omega_l-2e^{2\nvarphi}d\nvarphi\wedge\theta_l\\
    &=& e^{2\nvarphi}(\omega_l-2\gamma).
\end{eqnarray*}

The above geodesic spray $\Gamma^\h$ and the reduced nonholonomic vector field $\redvf$ are both quadratic sprays on $T(Q/G)$. Therefore, their difference is a vertical vector field $Z$,
$$
\Gamma^\h=\redvf+Z.
$$
If we can show that there exists a linear function $\tilde {P}$ such that $Z=\tilde{P}\Delta$, we know that there will exist a reparametrization between the base integral curves of both vector fields, and we may conclude the proof.

 With the above input, the characterizing equation $\iota_{\Gamma^\h}\omega_\h=dE_\h$ for $\Gamma^\h$ becomes 
    \begin{eqnarray*}
       e^{2\nvarphi}\iota_Z\omega_l-2e^{2\nvarphi}\iota_Z\gamma +e^{2\nvarphi}\iota_{\redvf}\omega_l-2e^{2\nvarphi}\iota_{\redvf}\gamma &=& 2e^{2\nvarphi}d\nvarphi l+e^{2\nvarphi}dE_l\\
       \Leftrightarrow\qquad \iota_Z (\omega_l-2\gamma)-\alphabar&=&2ld\nvarphi
    \end{eqnarray*}

We simply check that this uniquely determines $Z$ of the form $\tilde{P}\Delta$. Indeed, if we insert this we get
\begin{eqnarray*}
    \Leftrightarrow\qquad \tilde{P}\iota_{\Delta}(\omega_l-2\gamma)&=& 2ld\nvarphi+\alphabar\\
    \Leftrightarrow\quad \tilde{P}\iota_{\Delta}\omega_l-2\tilde{P}\iota_{\Delta}\gamma&=& 2ld\nvarphi +\iota_{\redvf}\gamma\\
   \Leftrightarrow \qquad\qquad -\tilde{P}\theta_l-0&=& 2ld\nvarphi+\iota_{\redvf}(d\nvarphi\wedge\theta_l)\\
  \Leftrightarrow \qquad\qquad\qquad -\tilde{P}\theta_l&=& 2ld\nvarphi+\redvf(\nvarphi)\theta_l-\iota_{\redvf}\theta_l d\nvarphi\\
  \Leftrightarrow \qquad\qquad\qquad-\tilde{P}\theta_l&=&2ld\nvarphi+\redvf(\nvarphi)\theta_l-2ld\nvarphi\\
   \Leftrightarrow \qquad\qquad\qquad-\tilde{P}\theta_l&=&\redvf(\nvarphi)\theta_l
\end{eqnarray*}
Therefore we may take $\tilde{P}=-\redvf(\nvarphi)$ and this will satisfy 
$$
\Gamma^\h=\redvf+Z=\redvf+\redvf(\nvarphi)\Delta.
$$
 From this, we may conclude that the base integral curves of the reduced equations are indeed pregeodesics of the metric $\h$.

(ii) In part (i) we relied on the assumption that $\iota_{\redvf}\gamma=\alphabar$.
When we have a projective geodesic extension of the form $(\overline{g}_{ai}=-G_{ai}, \nphi=\nvarphi\circ\pi)$,  conditions $(A')$ and $(B')$ are equivalent with the condition $(A')^G$. This condition is given by 
$
d\nvarphi\wedge\theta_l=\gamma^G.
$
We have already shown that the 1-form on the right hand side always satisfies $\iota_{\redvf}\gamma^G=\alphabar$. That is why, in this case, the extra condition in the statement of the Proposition will always be satisfied. 
\end{proof}

Remark that the "mathematical example" we constructed in Section~\ref{mathex} fits the requirements of Proposition~\ref{prop14} (ii), without being  $\varphi$-simple. Therefore, it is in fact an answer to a question at the end of \cite{Garcia}: to find examples of Hamiltonian reparametrizations of the reduced equations  which are not  $\varphi$-simple.

In the next Lemma, we give the coordinate expression of the relevant condition in Proposition~\ref{prop14}.
\begin{lem}
    The condition $\iota_{\redvf}\gamma=\alphabar$ is satisfied if and only if 
    $$ C_{bk}R_{ac}^k+C_{ak}R^k_{bc}=0,$$
    where $C_{bk}=\overline{g}_{bk}+G_{bk}$.
\end{lem}

\begin{proof} The condition is satisfied if, and only if, 
 \begin{eqnarray*}
     \frac{1}{6}(\overline{g}_{bk}R_{ac}^k-\overline{g}_{ck}R_{ab}^k+2\overline{g}_{ak}R^k_{bc})\dot{q}^a\dot{q}^b dq^c&=& -G_{ak} R_{bc}^k \dot{q}^a\dot{q}^b dq^c\\
     \Leftrightarrow \qquad \frac{1}{3}(\overline{g}_{bk}R_{ac}^k-\overline{g}_{ck}R_{ab}^k+2\overline{g}_{ak}R^k_{bc})\dot{q}^a\dot{q}^b &=& -G_{ak} R_{bc}^k\dot{q}^a\dot{q}^b\\
     \Leftrightarrow \qquad \frac{1}{3}(\overline{g}_{bk}R_{ac}^k-\overline{g}_{ck}R_{ab}^k+2\overline{g}_{ak}R^k_{bc} + \overline{g}_{ak}R_{bc}^k-\overline{g}_{ck}R_{ba}^k+2\overline{g}_{bk}R^k_{ac} ) &=& -G_{ak} R_{bc}^k-G_{bk}R_{ac}^k\\
     \Leftrightarrow \qquad  \overline{g}_{bk}R_{ac}^k+ \overline{g}_{ak}R_{bc}^k &=& -G_{ak} R_{bc}^k-G_{bk}R_{ac}^k\\
    \Leftrightarrow \qquad (\overline{g}_{bk}+G_{bk}   )R_{ac}^k+( \overline{g}_{ak}+G_{ak}   )R_{bc}^k&=&0\\
    \Leftrightarrow\qquad C_{bk} R_{ac}^k +C_{ak}R_{bc}^k&=&0
 \end{eqnarray*}
 And this is exactly what we wanted to show.
\end{proof}

As we had already announced in Section~\ref{sec:pge}, we end this paper by explaining how our approach using projective transformations of quadratic sprays (in the sense of adding a term along the Liouville vector field to form a projective class of sprays) relates to the approach that is used in e.g.\ \cite{Simoes}. We may also look at the reparametrizations  from another point of view, by considering the map 
$$
\psi:T(Q/G)\rightarrow T(Q/G),(q^a,\dot{q}^a)\mapsto (q^a,e^{-\nvarphi}\dot{q}^a).
$$
The next proposition, essentially says again that the reduced nonholonomic trajectories are reparametrizations of the geodesics of $\h$.
\begin{prop} \label{prop:last}
    Suppose that we have a projective geodesic extension  $(\overline{g}_{ai}, \nphi=\nvarphi\circ\pi)$ for a Chaplygin system. If $\gamma$ satisfies 
    $$
\iota_{\redvf}\gamma=\alphabar,
$$
the Hamiltonian vector field $e^{-\nvarphi}\redvf$ is $\psi$-related to the geodesic spray $\Gamma^\h$ of  the metric $\h$.
\end{prop}

\begin{proof}
If we define again $\h=e^{2\nvarphi}g_{ab}\dot{q}^a\dot{q}^b=e^{2\nvarphi}l$ and $\theta_\h=e^{2\nvarphi}g_{ab}\dot{q}^adq^b=e^{2\nvarphi}\theta_l$, then we have 
$$
\psi^* \h=e^{2\nvarphi}g_{ab}e^{-\nvarphi}\dot{q}^a e^{-\nvarphi}\dot{q}^b=l=E_l,
$$
$$
\psi^*\theta_\h=e^{2\nvarphi}g_{ab}e^{-\nvarphi}\dot{q}^a  dq^b=e^{\nvarphi}\theta_l
$$
and 
$$
\psi^*\omega_\h=-d(\psi^*\theta_\h)=-e^\varphi d\nvarphi\wedge\theta_l-e^{\varphi}d\theta_l=e^{\varphi}(\omega_l-\gamma)=e^{\nvarphi}\tilde{\tilde{\omega}}.
$$
Therefore, we may write the reduced system (again under the assumption  
$ \iota_{\redvf}\gamma=\alphabar$) as 
\begin{eqnarray*}
  \iota_{e^{-\nvarphi}\redvf} e^{\nvarphi}\tilde{\tilde{\omega}}&=&dE_l\\
  \Leftrightarrow\qquad \iota_{e^{-\nvarphi}\redvf}\psi^*\omega_\h&=&\psi^*(d\h)\\
 \Leftrightarrow \qquad (\psi^{-1})^*\Big[ \iota_{e^{-\nvarphi}\redvf}\psi^*\omega_\h&=&\psi^*(d\h) \Big]\\
 \Leftrightarrow \qquad \iota_{(\psi^{-1})^{*}(e^{-\nvarphi}\redvf)}\omega_\h&=&d\h.
\end{eqnarray*}

Since $\omega_\h$ is a symplectic form, we may conclude that 
$$
(\psi^{-1})^{*}(e^{-\nvarphi}\redvf)=\Gamma^\h\quad\Rightarrow\quad \psi_{*}(e^{-\nvarphi}\redvf)=\Gamma^\h.
$$
\end{proof}

In case ${\overline g}_{ai} = -G_{ai}$, we know that $e^{-\nvarphi}\redvf$ is a Hamiltonian vector field with respect to $e^{\nvarphi} \tilde{\tilde\omega}$, see the proof of Lemma~\ref{lem:closed} and Proposition~\ref{invvolformG}. In the even more  special case of $\varphi$-simplicity, the above proposition appears as Corollary~3.3  in \cite{Simoes}.

{\bf Acknowledgements.}  We thank the Research Fund of the University of Antwerp (BOF) for its support through the DOCPRO project 49747.

\end{document}